\documentclass[final,leqno,onefignum,onetabnum]{siamltex1213}

\usepackage{amsfonts}
\usepackage{graphicx}
\usepackage{amsopn}
\usepackage{amsmath}
\usepackage{comment}
\usepackage{amssymb}
\usepackage{pgfplots}

\newcommand{\R}{\mathbb{R}}

\newcommand{\ud}{\,{\rm d}}

\newcommand{\ti}{\partial_{t}}
\newcommand{\tr}{\partial_{r}}
\let\cd\cdot
\let\eps\varepsilon

\newcommand{\m}{\sqrt{M}}

\title{Radially symmetric shadow wave solutions to the system of pressureless gas dynamics in arbitrary dimensions}

\author{Marko Nedeljkov\thanks{Department of of Mathematics and Informatics, Faculty of Sciences, University of Novi Sad,
Trg Dositeja Obradovi\'{c}a 4,
21000 Novi Sad,
Serbia (marko@dmi.uns.ac.rs)}
  \and
Lukas Neumann\thanks{Unit of Engineering Mathematics, University of Innsbruck,
Technikerstra\ss e 13, 6020 Innsbruck,
Austria (lukas.neumann@uibk.ac.at)}
  \and
Michael Oberguggenberger\thanks{Unit of Engineering Mathematics, University of Innsbruck,
Technikerstra\ss e 13, 6020 Innsbruck,
Austria, (michael.oberguggenberger@uibk.ac.at)}
 \and
Manas R. Sahoo\thanks{School of Mathematical Sciences,
National Institute of Science Education and Research, Bhubaneswar,
P.O. Jatni, Khurda 752050, Odisha, India
(manas@niser.ac.in)}
}

\begin{document}
\maketitle

\begin{abstract}
Radially symmetric shadow wave solutions to the system of multidimensional pressureless gas dynamics are introduced, which allow one to capture concentration of mass. The transformation to a one-dimensional system with source terms is performed and physically meaningful boundary conditions at the origin are determined. Entropy conditions are derived and applied to single out physical (nonnegative mass) and dissipative (entropic) solutions. A complete solution to the pseudo-Riemann problem with initial data exhibiting a single delta shock on a sphere is obtained.
\end{abstract}



\pagestyle{myheadings}
\thispagestyle{plain}
\markboth{M. Nedeljkov, L. Neumann, M. Oberguggenberger, M. R. Sahoo}{Radially symmetric shadow waves}

\hyphenation{bounda-ry rea-so-na-ble be-ha-vior pro-per-ties cha-rac-te-ris-tic}

\section{Introduction}
This paper is devoted to radially symmetric nonclassical solutions to the multidimensional pressureless Euler system
\begin{equation}\label{EulerNd}
\begin{split}
\ti \rho+\nabla\cd\left(\rho \vec u\right)&=0\\
\ti\left(\rho\vec u\right)+\nabla\cd\left(\rho \vec u\otimes \vec u\right)&=0
\end{split}
\end{equation}
which consists of the convection part of the equations of conservation of mass and momentum in isentropic gas dynamics
\cite[Section 3.3]{Chen2003}. System \eqref{EulerNd} also describes the behavior of a gas consisting of sticky particles, i.e. particles that stick together if they collide.
This model and related models (sticky particles, adhesion particle dynamics, with or without viscosity) have a long history in cosmology, see e.g. the representative articles \cite{Shandarin, Vergassola, Weinberg}.

Following the seminal work of \cite{BrGr, E96}, the one-dimensional case has found wide\-spread attention. It has been evident from the beginning that solutions become measure-valued in finite time even for regular initial data. Thus solution concepts have been developed that admit accommodating nonclassical solutions such as delta shocks. We list a selection of results: The notion of duality solutions has been introduced in \cite{BoJa}. Further existence and uniqueness results on measure valued solutions have been obtained by various methods, including the mentioned duality approach \cite{Hu05,HuWang01,NgTu15}, particle approximations \cite{Grenier, NgTu}, zero viscosity limits \cite{Bo}, zero pressure limits \cite{Chen2011}, optimal transport theory \cite{Cavaletti}. For numerical methods, we refer to \cite{Bo94, BoJiLi, BoMa}. Another approach, namely shadow waves, is due to the first author. Shadow waves were introduced for the one-dimensional pressureless Euler system, among other conservation laws, in \cite{Ne10}, and will be the method of choice in this paper.

Concerning nonclassical solutions in the higher dimensional case, not so much seems to be known about system \eqref{EulerNd}, although this case is important in applications. There is work on radially symmetric solutions to related systems of magnetohydrodynamics \cite{Carot, Golovin, Stute, Tsui} and to the isentropic Euler equations \cite{LeFloch07}.
The numerically inspired weak asymptotic method has been applied to various systems of conservation laws, in particular to system \eqref{EulerNd}, in \cite{Colombeau}.
As general references for rotationally symmetric solutions to systems of conservation laws we mention \cite{Frei, Jenssen} and \cite{Hilgers}, the latter reference especially for the behavior at the origin.

In this paper, we address system \eqref{EulerNd} in $n$ space dimensions. The new contributions consist in
(a) extending the notion of shadow waves to the $n$-dimensional rotationally symmetric case; (b) constructing nonclassical radially symmetric solutions to the $n$-dimensional pressureless Euler system, in particular, solutions to the pseudo-Riemann problem, and (c) determining physically meaningful boundary conditions at the origin.
(The term pseudo-Riemann problem refers to initial data which are piecewise multiples of $|\vec x|^{1-n}$ with a jump at $|\vec x| = R > 0$; these arise due to the fact that the radially transformed system contains source terms.)

The plan of exposition is as follows: Section 2 addresses the radially symmetric version of system \eqref{EulerNd} and the appropriate boundary conditions at the origin. In Section 3, we recall the notion of shadow waves, extend it to the radially symmetric, higher dimensional case, and present the basic calculation of the weak limits fixing the structure of the shadow wave. This concerns both the local case (the behavior near the initial jump at $|\vec x| = R$) and the global behavior including the origin $\vec x = 0$.
Section 4 is a brief discussion of the case of solutions with constant wave speed (which will turn out to be the physically meaningful entropy solutions to the pseudo-Riemann problem). Section 5 is devoted to deriving and discussing entropy conditions for shadow wave solutions. Finally, in Section 6 we solve the pseudo-Riemann problem explicitly, obtaining solutions which are physical (i.e., with nonnegative density) and dissipative (i.e., satisfying the entropy condition). Further, these solutions are shown to be locally unique (for small time). In the last subsection we exhibit an example of nonuniqueness, provided by a global, physical solution which violates the entropy condition (and has non-constant wave speed).

\section{The radially symmetric pressureless Euler system}
As physical background, we consider the system of pressureless gas dynamics \eqref{EulerNd} as describing the evolution of the density $\rho(\vec x,t)$ and velocity field
$\vec u(\vec x,t)$ of a gas of sticky particles. Here $t\geq 0$ and $\vec x\in\R^n$; $\rho$ is a scalar quantity and $\vec u(\vec x,t) \in \R^n$.
Initial conditions $\rho(\vec x,0)=\rho_0(\vec x)$ and $\vec u(\vec x,0)=\vec u_0(\vec x)$ are presumed to be given.

Assuming that the initial data are radially symmetric it makes sense to search for radially symmetric solutions---indeed it is obvious that classical solutions are radially symmetric if the initial data are.
Making the ansatz $\rho(\vec x,t)=\rho(r,t)$ and $\vec u(\vec x,t)=u(r,t)\tfrac{\vec x}{|\vec x|}$, system \eqref{EulerNd} is transformed into
\begin{equation}\label{EulerNdRad}
\begin{split}
\ti \rho+\tr\left(\rho u\right)+\frac{n-1}{r}\rho u&=0\\
\ti \left(\rho u\right)+\tr\left(\rho  u^2 \right)+\frac{n-1}{r}\rho u^2&=0\,.
\end{split}
\end{equation}
These equations have to be accompanied by the initial data $\rho_0(r)$, $u_0(r)$. The system is non-strictly hyperbolic.
In principle, radially symmetric solutions can be constructed by using one-dimensional solutions and rescaling the density with the surface area of the corresponding sphere.
However, such an approach is not easily justifiable for nonclassical solutions because---as can be seen from the equations above---for non-smooth solutions also the nonlinear source terms will contain singularities.
Moreover, the transformation to polar coordinates introduces an artificial boundary at $r=0$. We want to mention, however, that for smooth solutions, system (\ref{EulerNdRad}) can be transformed to the one-dimensional pressureless Euler system (on the half line) by setting $u=r^{1-n}\tilde u$. The one-dimensional pressureless Euler system is well known to be equivalent to the Burgers equation, but again only for smooth solutions that in addition have strictly positive densities.
In view of these limitations,  we will not make use of such transformations---apart from maybe implicitly using them for regular parts of the solutions in Section \ref{SecRiem}.

%We note in passing that due to the relation
%$$\vec u(\vec x,t)= u(|\vec x|,t)\tfrac{\vec x}{|\vec x|} = \nabla\Big(\int_{0}^{|\vec x|}u(r,t)\ud r\Big)$$
%the flow of a radially symmetric velocity field is irrotational.

In the absence of pressure, it is possible that mass accumulates at $r=0$. We track the mass accumulating at $r=0$ and time $t$ in the value $m_0(t)$. In order to retain mass conservation we introduce the following boundary conditions
\begin{equation}\label{boundcond}
\begin{aligned}
&\text{for} &&u(0+,t)>0&& \dot m_0(t)=0\,,\ \rho(0+,t)=0 \,,\ u(0+,t) \text{ arbitrary}\\
&\text{for} &&u(0+,t)<0&& \dot m_0(t)=-\lim_{r\searrow 0}r^{n-1}|\mathbb{S}^{n-1}|\rho(r,t)u(r,t)\,,
\end{aligned}
\end{equation}
where $|\mathbb{S}^{n-1}|$ denotes the surface area of the sphere in $n$ dimensions and the quantity $u(0+,t)$ is to be understood, in the usual way, as the right-hand limit $\lim_{r\searrow 0}u(r,t)$.
Note that, as can be seen from the equations, in regions of zero mass the velocity can be prescribed arbitrarily---as one would also expect by physical considerations.
Moreover, due to the pressureless nature of the equation we cannot expect the Dirac delta function forming at zero to evaporate mass:
the mass accumulating at the origin will have zero velocity. This is expressed by the boundary conditions \eqref{boundcond}, and it is consistent with the interpretation as sticky particle system.

The total mass of a solution is given by $m(t)=m_0(t)+\int_0^\infty r^{n-1}|\mathbb{S}^{n-1}|\rho(r,t)\ud r$. The following formal calculation shows that mass conservation holds, at least for sufficiently regular solutions of compact support:
\begin{multline}\label{consmass}
\dot{m}(t)=\dot m_0(t)+\int_0^\infty r^{n-1}|\mathbb{S}^{n-1}|\,\ti \rho(r,t)\ud r\\
=\dot m_0(t)-\int_0^\infty r^{n-1}|\mathbb{S}^{n-1}|\left(\tr\left(\rho u\right)+\frac{n-1}{r}\rho u\right)\ud r\\
=\dot m_0(t)-\int_0^\infty |\mathbb{S}^{n-1}|\,\tr\left(r^{n-1}\rho u\right)\ud r=0\,,
\end{multline}
no matter whether the first or the second case in \eqref{boundcond} arises.

\section{Shadow wave solutions}\label{sect:swsolutions}
Shadow wave solutions are models of shock wave solutions that contain a singularity of the type of a delta function. In the radially symmetric case, they are supported on spheres and concern the density $\rho$. They were introduced for the one-dimensional pressureless Euler system, among other conservation laws, in \cite{Ne10}. Shadow waves are constructed as families of functions that approximate delta shock waves (or---as a matter of fact---other types of singular shocks) in a small $\eps$-neighborhood of the shock location. Outside the $\eps$-neighborhood, they are classical solutions of the system.
As $\eps\to 0$, their structure is enforced by requiring that conservation of mass and momentum holds in the weak sense. In the limit, a Dirac mass might be forming at the shock location. The limit is not viewed as an actual solution of the system in the algebraic sense, but captures the physical essence of conservation of mass and momentum flowing in or out of the singularity.

To deduce the conditions on the shock we first assume that the shock is located at position $R+c(t)>0$, where $c(0)=0$ and $R>0$. (Shadow waves emanating from the origin will not arise in the solution of the pseudo-Riemann problem.)
Using a simplification of one of the approaches from \cite{Ne10}, we make the ansatz
\begin{equation}\label{Ansatz}
U_{sw}^\eps(r,t)=\left(\rho_{sw}^\eps(r,t),u_{sw}^\eps(r,t)\right)=\left\{
 \begin{array}{ll}
  \left(\rho_0,u_0\right), &r-R<c(t)-\tfrac{\eps}{2}\\
  \left(\rho_{\eps},u_{\eps}\right), &c(t)-\tfrac{\eps}{2}<r-R<c(t)+\tfrac{\eps}{2}\\
   \left(\rho_1,u_1\right), &r-R>c(t)+\tfrac{\eps}{2}\,.
 \end{array}
\right.
\end{equation}
Here the singularity is located in the interval $R+c(t)-\epsilon/2,R+c(t)+\epsilon/2$ and travels with speed $\dot c(t)$.
All of the functions $\rho_0$, $\rho_\eps$, $\rho_1$, $u_0,$, $u_\eps$, $u_1$ are assumed to be $C^1$-regular and may depend on $r$ and $t$, in general.
For the moment we assume $R+c(t)>0$, avoiding interaction of the delta wave with the singularity at zero. In other words, we are first interested in the local behavior of shadow waves. The global behavior, including boundary conditions at the origin, will be addressed afterwards.

Before arriving at a formal definition of a shadow wave solution, we dedicate the next paragraphs to studying the behavior of shadow waves for fixed $\eps > 0$ and as $\eps\to 0$. We are guided by the idea that a shadow wave solution should have a singularity in the mass of the type of a Dirac delta function, and at the same time satisfy conservation of mass and momentum, at least in the weak sense as $\eps\to 0$. We make the additional simplification that
\begin{equation}\label{simplification}
   \rho_{\eps} = \frac1{\eps}\sigma(t),\quad u_{\eps} = v(t)
\end{equation}
for some $C^1$-regular functions $\sigma$ and $v$,
and plug the ansatz (\ref{Ansatz}) into the weak formulation of the equation, that is
\begin{align}
\int_0^\infty \int_0^\infty \left[-\rho\ti\phi -\left(\rho u\right)\tr\phi +\frac{n-1}{r}\rho u\phi\right] \ud r\ud t&=0\label{weakmass}\\
\int_0^\infty \int_0^\infty \left[- \left(\rho u\right)\ti\phi-\left(\rho  u^2 \right)\tr \phi+\frac{n-1}{r}\rho u^2\phi\right] \ud r\ud t&=0\label{weakmomentum}
\end{align}
for test functions $\phi\in\mathcal{D}\left(\mathbb{R}\times[0,\infty[\right)$.
We only perform the calculations of the terms in \eqref{weakmass} here and will comment on equation \eqref{weakmomentum} later.

First observe that the following integration by parts rules hold (for $C^1$-functions $f,g$)
\begin{align*}
\iint_{r\leq R+c(t)}\partial_t \left(f(r,t)g(r,t)\right)\ud (r,t)&=-\int_0^\infty(fg)(R+c(t),t)\,\dot c(t)\ud t\\
\iint_{r\leq R+c(t)}\partial_r \left(f(r,t)g(r,t)\right)\ud (r,t)&=\int_0^\infty(fg)(R+c(t),t)\,\ud t\,.
\end{align*}
These can be derived, for example, by using the two-dimensional divergence theorem on the vector fields $(fg,0)^T$ and $(0,fg)^T$, respectively.

Using these formulas in (\ref{weakmass}) at fixed $\eps > 0$ and evaluating the first term yields
\begin{multline*}
-\int\int \rho_{sw}^\eps(r,t)\partial_t\phi(r,t)\ud r\ud t = \int_0^\infty \int_0^{R+c(t)-\epsilon/2}\partial_t \rho_0\phi\, \ud r\ud t\\
+\int_0^\infty \int_{R+c(t)-\epsilon/2}^{R+c(t)+\eps/2}\tfrac1{\eps}\partial_t{\sigma}\phi\,\ud r\ud t+
\int_0^\infty \int_{R+c(t)+\epsilon/2}^\infty \partial_t \rho_1\phi\,\ud r\ud t\\
+\int_0^\infty \Big((\rho_0\phi)(R+c(t)-\tfrac{\eps}{2},t)-\tfrac1{\eps}{\sigma}(t)\phi(R+c(t)-\tfrac{\eps}{2},t)\Big)\dot c(t)\ud t\\
+\int_0^\infty \Big(-(\rho_1\phi)(R+c(t)+\tfrac{\eps}{2},t)+\tfrac1{\eps}{\sigma}(t)\phi(R+c(t)+\tfrac{\eps}{2},t)\Big)\dot c(t)\ud t\,,
\end{multline*}
while the term containing spatial derivatives becomes
\begin{multline*}
-\int\int \left(\rho_{sw}^\eps u_{sw}^\eps\right)(r,t)\partial_r\phi(r,t)\ud r\ud t\\
=\int_0^\infty \int_0^{R+c(t)-\epsilon/2}\partial_r\left(\rho_0 u_0\right)\phi\, \ud r\ud t+
\int_0^\infty \int_{R+c(t)+\epsilon/2}^\infty \partial_r\left(\rho_1 u_1\right)\phi\,\ud r\ud t\\
-\int_0^\infty \Big((\rho_0 u_0\phi)(R+c(t)-\tfrac{\eps}{2},t)-\tfrac1{\eps}{\sigma}(t)v(t)\phi(R+c(t)-\tfrac{\eps}{2},t)\Big)\ud t\\
-\int_0^\infty \Big(-(\rho_1 u_1\phi)(R+c(t)+\tfrac{\eps}{2},t)+\tfrac1{\eps}{\sigma}(t)v(t)\phi(R+c(t)+\tfrac{\eps}{2},t)\Big)\ud t\,.
\end{multline*}
We use Taylor expansion of the test function $\phi$ in $r=R+c(t)$,
\begin{align*}
 \phi(R+c(t)-\tfrac{\eps}{2},t)&=\phi(R+c(t),t)-\tfrac{\eps}{2}\partial_r\phi(R+c(t),t)+\mathcal{O}(\eps^2)\,,\\
 \phi(R+c(t)+\tfrac{\eps}{2},t)&=\phi(R+c(t),t)+\tfrac{\eps}{2}\partial_r\phi(R+c(t),t)+\mathcal{O}(\eps^2)\,,\\
 \phi(r,t)&=\phi(R+c(t),t)+\mathcal{O}(\eps),\  R+c(t)-\tfrac{\eps}{2} < r < R+c(t)+\tfrac{\eps}{2}
\end{align*}
to calculate the contributions from the boundary integrals as well as the contribution from integrating about the singularity.
We find
\begin{multline*}
 -\int\int \rho_{sw}^\eps(r,t)\partial_t\phi(r,t)\ud r\ud t=\int_0^\infty \int_0^{R+c(t)-\epsilon/2}\partial_t \rho_0\phi\, \ud r\ud t\\
+\int_0^\infty \partial_t \sigma(t)\phi(R+c(t),t)\ud t+\int_0^\infty \int_{R+c(t)+\epsilon/2}^\infty \partial_t \rho_1\phi\,\ud r\ud t\\
+\int_0^\infty \Big(\rho_0(R+c(t)-\tfrac{\eps}{2},t)-\rho_1(R+c(t)+\tfrac{\eps}{2},t)\Big)\phi(R+c(t),t)\dot c(t)\ud t\\
+\int_0^\infty \sigma(t)\partial_r\phi(R+c(t),t)\dot c(t)\ud t+\mathcal{O}(\eps)
\end{multline*}
and
\begin{multline*}
-\int\int \left(\rho_{sw}^\eps u_{sw}^\eps\right)(r,t)\partial_r\phi(r,t)\ud r\ud t\\
=\int_0^\infty \int_0^{R+c(t)-\epsilon/2}\partial_r\left(\rho_0 u_0\right)\phi\, \ud r\ud t
+\int_0^\infty \int_{R+c(t)+\epsilon/2}^\infty \partial_r\left(\rho_1 u_1\right)\phi\,\ud r\ud t\\
+\int_0^\infty \Big(-(\rho_0 u_0)(R+c(t)-\tfrac{\eps}{2},t)+(\rho_1 u_1)(R+c(t)+\tfrac{\eps}{2},t)\Big)\phi(R+c(t),t)\ud t\\
-\int_0^\infty \sigma(t)v(t)\partial_r\phi(R+c(t),t)\ud t+\mathcal{O}(\eps)\,.
\end{multline*}
Concerning the source term, we note that
\begin{align*}
   \lim_{\eps\to 0} \int_0^\infty \int_{R+c(t)-\epsilon/2}^{R+c(t)+\eps/2} \tfrac{n-1}{r\eps}&\sigma(t)v(t)\phi(r,t)\ud r\ud t\\
      &= \int_0^\infty \tfrac{n-1}{R+c(t)}\sigma(t)v(t)\phi(R+c(t),t)\ud t\,.
\end{align*}
From the integral equation \eqref{weakmass} and the subsequent calculations we obtain the following conditions as $\eps\rightarrow0$:
\\
First, from the double integrals, that is terms away from the front, we find that
 \begin{equation*}
 \begin{aligned}
 &(\rho_0,u_0)\text{ is a solution to }\eqref{weakmass} \text{ in }\\
 &(\rho_1,u_1)\text{ is a solution to }\eqref{weakmass} \text{ in }
  \end{aligned}
 \begin{aligned}
 &]0,R+c(t)[\times ]0,\infty[\\
 &]R+c(t),\infty[\times ]0,\infty[
  \end{aligned}
 \end{equation*}
Second, balancing the single integrals, that is terms with support near the front, and using that the terms multiplying $\phi$ and $\partial_r\phi$ have to satisfy the equation independently, we get
\begin{align}
   -\dot c(\rho_1-\rho_0)+(\rho_1 u_1-\rho_0 u_0)+\left(\partial_t\sigma+\tfrac{n-1}{R+c(t)}\sigma v \right) &= 0\ \ \text{and} \label{swcond1}\\
   \dot c(t) \sigma(t) - \sigma(t) v(t) &= 0\,. \label{swcondvelo1}
  \end{align}
We do not perform the calculations for equation \eqref{weakmomentum} explicitly here, but along the same lines one finds that
$(\rho_0,u_0)$ resp. $(\rho_1,u_1)$ have to be solutions to \eqref{weakmomentum} in the corresponding regions and
\begin{align}
   -\dot c(\rho_1 u_1-\rho_0 u_0)+(\rho_1 u_1^2-\rho_0 u_0^2)+
   \left(\partial_t(\sigma v)+\tfrac{n-1}{R+c(t)}\sigma v^2 \right) &= 0\,,\label{swcond2}\\
   \dot c(t) \sigma(t) v(t) - \sigma(t) v^2(t) &= 0\,. \label{swcondvelo2}
\end{align}
We introduce the standard notation for the jump in density $[\rho]=\rho_1-\rho_0$ and momentum $[\rho u]=\rho_1 u_1-\rho_0 u_0$ and use the abbreviations
\begin{equation}\label{kappa12}
\begin{array}{l}
 \kappa_1 :=\, \dot c[\rho]-[\rho u]\\[4pt]
 \kappa_2 :=\, \dot c[\rho u]-[\rho u^2]
 \end{array}
\end{equation}
where it is understood that $\kappa_1$ and $\kappa_2$ are functions of $R+c(t)$.
Note that $\kappa_1$ is the mass flowing into the shadow wave.
The first part is the net influx due to the movement of the shadow wave and the second part corresponds to the net flux by differences in velocity and density to the left and right of the front.
Similarly $\kappa_2$ can be interpreted as the net momentum flux into the shadow wave.

With these notations, equations \eqref{swcond1} and \eqref{swcond2} can be written as the differential equations
\begin{align}
\partial_t\sigma+\tfrac{n-1}{R+c(t)}\sigma v  &= \kappa_1\,, \label{swdiff1}\\
    \partial_t(\sigma v)+\tfrac{n-1}{R+c(t)}\sigma v^2 &= \kappa_2\,.\label{swdiff2}
\end{align}
Further, \eqref{swcondvelo1} implies \eqref{swcondvelo2} and is satisfied if $v(t) = \dot{c}(t)$ or $\sigma(t) = 0$.
%Thus the front speed is actually given by the value $u_\eps$ of $u$ in the front.

Summarizing the calculations above, we arrive at the following assertion.

\begin{proposition}\label{swprop}
A family of functions  $U_{sw}^\eps$ of the form  \eqref{Ansatz} satisfies system \eqref{EulerNdRad} in the weak limit as $\eps\to 0$ if and only if
\begin{itemize}
\item[(1)] $(\rho_0,u_0)$ is a weak solution to \eqref{EulerNdRad} in $]0,R+c(t)[\times ]0,\infty[$,
\item[(2)] $(\rho_1,u_1)$ is a weak solution to \eqref{EulerNdRad} in $]R+c(t),\infty[\times ]0,\infty[$ and
\item[(3)] $c$, $\sigma$ and $v$ are related through the differential equations  \eqref{swdiff1}, \eqref{swdiff2}.
\end{itemize}
Further, if $\sigma(t) \neq 0$ on some time interval, then necessarily $v(t) \equiv \dot{c}(t)$ there.
\end{proposition}

\begin{definition}\label{def:sw}
(a) A family of functions $U_{sw}^\eps$ of the type introduced in \eqref{Ansatz} is called a \emph{local shadow wave (SW) solution} for system \eqref{EulerNdRad}, if the conditions (1), (2) and (3) of Proposition \eqref{swprop} are satisfied and $v(t) \equiv \dot{c}(t)$.\\
% We impose, in addition, that $\sigma(0)=0$ to avoid concentrations of mass already in the initial data.
(b) A local shadow wave is called \emph{physical} if the mass in the shadow wave $\sigma(t)$ is nonnegative on its interval of existence.
\end{definition}

{\sc Remark.} (a) Let us briefly comment on the special case  $\sigma(t) = 0$. Due to \eqref{swdiff1}, \eqref{swdiff2}, this case can only arise for $t$ in some time interval if $\kappa_1 = 0$ and $\kappa_2 = 0$ on that time interval. But $\kappa_1 = \kappa_2 = 0$ are just the Rankine-Hugoniot conditions for piecewise differentiable solutions. Together with $\sigma = 0$ this means that the limit of $U_{sw}^\eps$ is a shock wave solution of system  \eqref{EulerNdRad}. In addition, a simple calculation using \eqref{kappa12} shows that  $\kappa_1 = \kappa_2 = 0$ implies that $\rho_0 =0$ or $\rho_1 = 0$ or $u_0 = u_1$. In these three cases, the resulting shock speeds are easily calculated to be $\dot{c} = u_1$, $\dot{c} = u_0$ and $\dot{c} = u_0 = u_1$, respectively.

(b) In ansatz \eqref{Ansatz} one could admit intermediate functions $(\rho_\eps, u_{\eps})$ of a more general form than \eqref{simplification}. As long as
$\rho_\eps=\mathcal{O}(\eps^{-1})$, $u_\eps = \mathcal{O}(1)$ and sufficiently strong uniformity conditions (in $\eps$, $r$ and $t$) are satisfied, the same calculations as above would lead to the existence of the limits $\lim_{\eps\rightarrow 0}\eps \rho_\eps(R+c(t),t)=\sigma(t)$, $\lim_{\eps\rightarrow 0}u_\eps(R+c(t),t)=v(t)$ where $\sigma$ and $v$ satisfy the conditions of Proposition \eqref{swprop}. Thus qualitatively no new solutions are gained by this generalization. Also, the  \emph{two-sided} shadow waves with different values of
$(\rho_\eps, u_{\eps})$  to the left and right of the curve $r = R+c(t)$ will not lead to qualitatively different shadow wave solutions in the context of pressureless gas dynamics.

(c) Here and in the remainder of this section we denote the regular, piecewise $C^1$-part of the shadow wave by $\rho(r,t)$ and $u(r,t)$, that is
\begin{equation}\label{regpart}
(\rho(r,t), u(r,t)) = \left\{\begin{array}{ll}
                                  (\rho_0(r,t), u_0(r,t)), & r < R + c(t)\\
                                  (\rho_1(r,t), u_1(r,t)), & r > R + c(t).
				\end{array}\right.
\end{equation}
As $\eps\to 0$, the local shadow wave solution of equation~\eqref{EulerNd} converges to a delta shock wave of the form
 \begin{equation*}
 \begin{split}
  \rho_{\rm lim}(\vec x,t)&=\rho(|\vec x|,t)+m_0(t)\delta_{\vec x=0}+|\mathbb{S}^{n-1}|\sigma(t)\delta_{|\vec x|=R+c(t)}\ \  \text{and}\\
  \vec u_{\rm lim}(\vec x,t)&=u(|\vec x|,t)\tfrac{\vec x}{|\vec x|}\ ,\ \text{for}\  |\vec{x}| > 0\,.
%\vec u(\vec x,t)&=u(|\vec x|,t)\tfrac{\vec x}{|\vec x|}\ ,\ \text{where}\  u(0,t)=0 \text{ and } u(R+c(t),t)=\dot c(t)\,,
 \end{split}
 \end{equation*}
Indeed, the assertion is clear for the regular parts. The singular part is handled by the following calculation:
\begin{multline}\label{swdeltacalc}
 \int_{\mathbb{S}^{n-1}}\int_{R+c(t)-\eps/2}^{R+c(t)+\eps/2} \tfrac1{\eps}\sigma(t)\,r^{n-1}\ud r\phi(\vec\omega r)\ud \vec \omega \\
 = \int_{\mathbb{S}^{n-1}}\int_{-1/2}^{1/2} \sigma(t)(\eps r+R+c(t))^{n-1}\phi(\vec\omega (\eps r+R+c(t)))\ud r)\ud \vec\omega\\
 \longrightarrow |\mathbb{S}^{n-1}|\sigma(t)(R+c(t))^{n-1}\phi(\vec\omega (R+c(t)))=\\
   = |\mathbb{S}^{n-1}|\sigma(t)(R+c(t))^{n-1}\langle \delta_{|\vec x|=R+c(t)},\phi\rangle\,.
\end{multline}

We now turn to the global case including $r=0$ and state what we mean by a radially symmetric shadow wave solution, specifying the required behavior at the origin.
\begin{definition}\label{def:radialSW}
(a) A triple $\left(\rho_{sw}^\eps(r,t),u_{sw}^\eps(r,t),m_0(t)\right)$ is called a \emph{radially symmetric shadow wave solution} of system~\eqref{EulerNd}, if
\begin{itemize}
\item[(1)] away from the origin, $\left(\rho_{sw}^\eps(r,t),u_{sw}^\eps(r,t)\right)$ is a local shadow wave in the sense of Definition~\ref{def:sw};
\item[(2)] the limits $\lim_{r\searrow 0}u(r,t)$ and $\lim_{r\searrow 0}\rho(r,t)r^{n-1}$ exist, except possibly at a discrete set of times $t$;
\item[(3)] the boundary conditions \eqref{boundcond} hold.
\end{itemize}
Moreover, in case there exists $t_0>0$ such that $R+c(t_0)=0$, it is required that $m_0(t_0+)=m_0(t_0-)+\lim_{t\nearrow t_0}\sigma(t)\left(R+c(t)\right)^{n-1}|\mathbb{S}^{n-1}|$\,.

(b) If system~\eqref{EulerNd} is supplemented by initial data $\left(\rho^0(r),u^0(r)\right)$, locally integrable on $\{r \geq 0\}$, a radially symmetric shadow wave solution is required to satisfy
$\lim_{\eps\to 0} \left(\rho_{sw}^\eps(\cdot,t),u_{sw}^\eps(\cdot,t)\right) = \left(\rho^0(r),u^0(r)\right)$ in the sense of distributions on $\{r > 0\}$, and $m(0+)= 0$.
\end{definition}

{\sc Remark.}
(a) Conservation of mass: Let  $\left(\rho_{sw}^\eps(r,t),u_{sw}^\eps(r,t),m_0(t)\right)$ be a radially symmetric shadow wave solution of system~\eqref{EulerNd} with compact support in the space variable, and let $\rho$ be the regular part \eqref{regpart} of its distributional limit.
Then the quantity
\begin{equation}\label{totalmass}
 Q(t)= m_0(t) + \int_0^{\infty} r^{n-1} \lvert S^{n-1}\rvert {\rho}(r,t)\ud r
     +\lvert S^{n-1}\rvert(R+c(t))^{n-1}\sigma(t)
\end{equation}
is conserved in the evolution. Here the first term
\[
    m_0(t) = -\int_{0}^{t} \displaystyle {\lim_{r\searrow 0} (\lvert S^{n-1} \rvert r^{n-1}{\rho}(r,s)u(r,s))}\ud s
\]
captures the mass accumulating at the origin, where the limit in the integrand is understood to be zero when the first case of the boundary condition \eqref{boundcond} occurs. The second term in $Q(t)$ describes the mass of the regular part of the solution, while the third term depicts the mass concentrating along the location of the singular part of the shadow wave.

Clearly, conservation of mass was one of the basic ingredients in deriving the structure of the shadow wave. Nevertheless, it is instructive to check that the resulting distributional limit obeys mass conservation, indeed.
We write
\begin{equation*}
 \begin{aligned}
 Q(t)\ =\ & m_0(t) + \int_0^{R+c(t)} r^{n-1} \lvert S^{n-1}\rvert {\rho}(r,t)\ud r
    +\int_{R+c(t)}^{\infty} r^{n-1} \lvert S^{n-1}\rvert{\rho}(r,t)\ud r\\
   & +\ \lvert S^{n-1}\rvert(R+c(t))^{n-1}\sigma(t)\,.
 \end{aligned}
 \end{equation*}
Differentiating $ Q(t)$ with respect to $t$ yields
\begin{equation*}
 \begin{aligned}
 \dot{Q}(t)\ =\ & \dot{m}_0(t) + \int_0^{R+c(t)} r^{n-1} \lvert S^{n-1}\rvert\ti {\rho}(r,t)\ud r
   +\int_{R+c(t)}^{\infty} r^{n-1} \lvert S^{n-1}\rvert\ti{\rho}(r,t)\ud r\\
 &-\ (R+c(t))^{n-1}\lvert S^{n-1}\rvert [\rho]\dot{c}(t)\\[2pt]
 &+\ (n-1)\lvert S^{n-1}\rvert(R+c(t))^{n-2} \sigma(t)\dot{c}(t) + \lvert S^{n-1}\rvert (R+c(t))^{n-1}\dot{\sigma}(t)\,.
 \end{aligned}
 \end{equation*}
Using the fact that the regular part satisfies \eqref{EulerNdRad} we get, as in \eqref{consmass}, that
\begin{equation*}
 \begin{aligned}
 \dot{Q}(t)\ =\
 & \dot{m}_0(t) - \int_0^{R+c(t)}  \lvert S^{n-1}\rvert  \partial_{r}(r^{n-1} \rho u) \ud r
    - \int_{R+c(t)}^{\infty}  \lvert S^{n-1}\rvert \partial_{r}(r^{n-1} \rho u) \ud r\\
 &-\ (R+c(t))^{n-1}\lvert S^{n-1}\rvert [\rho]\dot{c}(t)\\
 &+\ (n-1)\lvert S^{n-1}\rvert(R+c(t))^{n-2} \sigma(t)\dot{c}(t) + \lvert S^{n-1}\rvert \dot{\sigma}(t)(R+c(t))^{n-1}\\[2pt]
 =\ &   \lvert S^{n-1}\rvert (R+c(t))^{n-1} [\rho u]
 - (R+c(t))^{n-1}\lvert S^{n-1}\rvert [\rho]\dot{c}(t)\\
 &+\ (n-1)\lvert S^{n-1}\rvert(R+c(t))^{n-2} \sigma(t)\dot{c}(t) + \lvert S^{n-1}\rvert \dot{\sigma}(t)(R+c(t))^{n-1}\\
 =\ &   \lvert S^{n-1}\rvert (R+c(t))^{n-1}\Big( [\rho u]
 -[\rho]\dot{c}(t)
 +\frac{n-1}{R+c(t)} \sigma(t)\dot{c}(t) +\dot{\sigma}(t)\Big)=0.
 \end{aligned}
 \end{equation*}

(b) Conservation of momentum: In a similar way, one can show that the evolution conserves the quantity
\begin{equation*}
 \begin{aligned}
 M(t)=&-\int_{0}^{t} \displaystyle {\lim_{r\searrow 0} (\lvert S^{n-1} \rvert r^{n-1}{\rho}(r,s)u^2(r,s))}\ud s + \int_0^{\infty} r^{n-1} \lvert S^{n-1}\rvert {\rho}(r,t)u(r,t)\ud r\\
 &+\lvert S^{n-1}\rvert(R+c(t))^{n-1}\sigma(t)\dot{c}(t).
 \end{aligned}
 \end{equation*}

\section{Existence of shadow wave solutions with constant wave speed}\label{sec:constspeed}
Recall from the definition of a shadow wave that the mass term $\sigma(t)$ and the velocity $c(t)$ must satisfy
\begin{align}
   \partial_t \sigma+\tfrac{n-1}{R+c(t)}\dot c \sigma&= \kappa_1=\dot c[\rho]-[\rho u]\label{equ:sw:def1}\\
  \partial_t\left(\dot c \sigma\right)+\tfrac{n-1}{R+c(t)}\dot c^2 \sigma&= \kappa_2=\dot c[\rho u]-[\rho u^2]\,,\label{equ:sw:def2}
  \end{align}
where $\kappa_1$ and $\kappa_2$ are evaluated at $r = R+c(t)$.
Using the product rule in equation~\eqref{equ:sw:def2} and inserting the result in \eqref{equ:sw:def1} leads to
  \begin{equation}\label{equ:sw:def22}
  \sigma \ddot c+\dot c\kappa_1=\kappa_2\,.
  \end{equation}
In this section we consider the case when $\ddot c=0$, that is the wave speed is constant, $\dot c(t)=v_0=const$, and $c(t)=tv_0$. Equations \eqref{equ:sw:def1} and \eqref{equ:sw:def22} become
\begin{align*}
  \partial_t \sigma+\tfrac{n-1}{R+v_0t}v_0 \sigma&= \kappa_1=v_0[\rho]-[\rho u]\\
 v_0\kappa_1&=\kappa_2=v_0[\rho u]-[\rho u^2]\,.
\end{align*}
We impose the initial condition $\sigma(0) = 0$, corresponding to the generic case that there is no concentration of mass initially.

Let us for now assume that $\kappa_2\neq0$ and thus $v_0$ as well as $\kappa_1$ have to be different from $0$. The case of a standing shadow wave shall be considered later.
The homogenous solution of the first equation is given by $\sigma(t)=const\cdot (R+v_0t)^{1-n}$ where the constant has to be nonnegative for physical shadow waves.
To solve the inhomogeneous problem we use variation of constants and the fact that for constant $\dot c$ we have $\kappa_1=\kappa_1(R+v_0t)$.
This results in
\begin{equation}\label{SWMass}
  \begin{array}{rcl}
 \sigma(t) &=&  \displaystyle\int_{0}^{t}\kappa_1(R+v_0s)\big(R+v_0s\big)^{n-1}\ud s\cdot\big(R+v_0t\big)^{1-n}\\[8pt]
           & = &\displaystyle\frac{1}{v_0}\int_{R}^{R+v_0 t}\kappa_1(r)r^{n-1}\ud r\cdot\big(R+v_0t\big)^{1-n}\,.
 \end{array}
\end{equation}
This solution satisfies the initial condition $\sigma(0)=0$.
To find $v_0$ we have to solve the quadratic equation $v_0\kappa_1=\kappa_2$. (Note that $\kappa_1$ and $\kappa_2$ may be functions of $r = R + c(t)$, but
$\kappa_2/\kappa_1$ must be constant.) For $[\rho]\neq 0$ we obtain the solutions
\begin{equation}\label{ConstWS:speeds}
 _1v_0=\frac{u_1\sqrt{\rho_1}+u_0\sqrt{\rho_0}}{\sqrt{\rho_1}+\sqrt {\rho_0}}\quad \text{or}\quad _2v_0=\frac{u_1\sqrt{\rho_1}-u_0\sqrt{\rho_0}}{\sqrt{\rho_1}-\sqrt {\rho_0}}\,.
\end{equation}
For $[\rho]=0$, there is the single solution $v_0 = (u_0+u_1)/2$ which coincides with $ _1v_0$. We will later see that the second root $_2v_0$ never leads to an entropy solution and thus physical shadow waves have the unique propagation velocity $_1v_0$.
%%%%%%%%%%%%%%%%%%%%%%%%%%%%%%%%%%%%%%%%%%%%%%%%%%%%%%%%%%%%%%%%%%%%%%%%%%%%%%%%%%%%%%%%%%%%%%%%%%%%%%%%%%%%%%%%%%%%%%%%%%%%%%%%%%%%%%%%%%%%%%%%%%%%%%%%%%

%%%%%%%%%%%%%%%%%%%%%%%%%%%%%%%%%%%%%%%%%%%%%%%%%%%%%%%%%%%%%%%%%%%%%%%%%%%%%%%%%%%%%%%%%%%%%%%%%%%%%%%%%%%%%%%%%%%%%%%%%%%
\begin{proposition}\label{ConstWS:exist1}
Assume that $v_0$ is one of the (constant) roots given in \eqref{ConstWS:speeds} and that $\kappa_2\neq 0$.
Then a shadow wave solution of system \eqref{EulerNdRad} is given by
\begin{equation}\label{SwconstV}
 \big(\rho_{sw}^\eps(r,t),u_{sw}^\eps(r,t)\big)=\left\{
   \begin{aligned}
                      &(\rho_0, u_0)\\
         &\Big(\tfrac{1}{\eps v_0}\int_{R}^{R+v_0 t}\kappa_1(r)r^{n-1}\ud r\cdot\big(R+v_0t\big)^{1-n}, v_0\Big)\\
                      &(\rho_1,u_1)
   \end{aligned}\right.
\end{equation}
with the three expressions on the right-hand side valid in the regions $r < R+v_0t-\frac{\eps}2$,
$R+v_0t-\frac{\eps}2 < r < R+v_0t+\frac{\eps}2$ and
$r > R+v_0t+\frac{\eps}2$, respectively. Further, $(\rho_0,u_0)$ and $(\rho_1,u_1)$ are classical solutions in the corresponding regions.
\end{proposition}

Now we turn to the case $\kappa_2=0$.
In this case either $\kappa_1=0$ or $v_0=0$. As noted in the remark after Definition \ref{def:sw}, the case $\kappa_1 = \kappa_2$ leads to $\sigma(t)\equiv 0$ and to a shock wave rather than a shadow wave with a delta function part. So we are left with the case $v_0=\kappa_2=0$.
Then by the definition of $\kappa_1$ and $\kappa_2$ we have $\kappa_1=-[\rho u]$ and $0=\kappa_2=-[\rho u^2]$. Further, $\partial_t \sigma=\kappa_1$ and we arrive at the following result:
\begin{proposition}\label{ConstWS:exist2}
Assume that  $\kappa_2=0$. Then a shadow wave solution of system \eqref{EulerNdRad} is given by
\begin{equation}\label{Swrest}
\rho_{sw}^\eps(r,t)=\left\{
   \begin{aligned}
                      &\rho_0\\
                      &-\tfrac{t}{\eps}[\rho u]\\
                      &\rho_1
   \end{aligned}\right.\, ,\quad
  u_{sw}^\eps(r,t)=\left\{
   \begin{aligned}
                      &u_0\quad&&r<R-\tfrac{\eps}2\\
                      &0&&R-\tfrac{\eps}2<r<R+\tfrac{\eps}2\\
                      &u_1&&r>R+\tfrac{\eps}2\, .
   \end{aligned}\right.
\end{equation}
\end{proposition}
Note that there exists no shadow wave solution of constant wave speed if $\kappa_2\neq 0$ but $\kappa_1=0$. However, this is the non-physical case of momentum flowing into the singularity without mass flux.

\section {Admissibility conditions and overcompressibility}

Entropy conditions and---in the case of delta shocks---overcompressibility are commonly used to single out physically meaningful solutions. This section starts with a formal derivation of entropy inequalities (for sufficiently smooth solutions). Requiring that the same inequalities hold in the weak limit will produce the notion of a \emph{dissipative shadow wave solution}.

\subsection{The kinetic energy as entropy functional}
Considering the decay of kinetic energy in the formation of the shadow wave the total kinetic energy appears as a reasonable entropy functional for the pressureless Euler dynamics. The kinetic energy is given by
\[
 E(\rho,\vec u):=\tfrac{1}{2}\rho|\vec u|^2\,.
\]
In order to establish entropy fluxes we rewrite the system in canonical form. Setting $m_i=\rho u_i$ for $i\in\{1,2,3\}$ and $U=(\rho,m_1,m_2,m_3)$ we can rewrite, following \cite[Section 7.2]{Da},
equation \eqref{EulerNd} in the canonical form
\[
 \partial_t U+\nabla_{\vec x} F(U)=0
\]
where $F$ is defined as
\[
 F\big(\rho,m_1,m_2,m_3\big)=
   \begin{pmatrix}
    m_1+m_2+m_3\\[2pt]\tfrac{1}{\rho}(m_1^2+m_1m_2+m_1m_3)\\[4pt]\tfrac{1}{\rho}(m_2m_1+m_2^2+m_2m_3)\\[4pt]\tfrac{1}{\rho}(m_3m_1+m_3m_2+m_3^2)
    \end{pmatrix}\,.
\]
Now we can split $F$ by setting $F=\sum_{j=1}^3 G_j$, where
\[
G_j\big(\rho,m_1,m_2,m_3\big) = \tfrac{m_j}{\rho}
   \begin{pmatrix}
    \rho\\m_1\\m_2\\ m_3
    \end{pmatrix}\,.
\]
The integrability conditions for the existence of an entropy--entropy-flux pair with entropy $E$ then can be expressed \cite[Section 7.4]{Da} as
\begin{equation}\label{eq:entr:integrcond}
\mathrm{D}^2E\,\mathrm{D}G_j=\mathrm{D}G_j^T\,\mathrm{D}^2E\,.
\end{equation}
We only check this condition for $j=1$, the other relations can be checked in a similar fashion. We have
\[
\mathrm{D}^2E=\mathrm{D}^2 \left(\tfrac{|\vec m|^2}{2\rho}\right) =
 \begin{pmatrix}
 \tfrac{m_1^2+m_2^2+m_3^2}{\rho^3} & -\tfrac{m_1}{\rho^2} & -\tfrac{m_2}{\rho^2} & -\tfrac{m_3}{\rho^2}\\[4pt]
 -\tfrac{m_1}{\rho^2} & \tfrac{1}{\rho} & 0 & 0\\[4pt]
 -\tfrac{m_2}{\rho^2} & 0 &\tfrac{1}{\rho} & 0\\[4pt]
 -\tfrac{m_3}{\rho^2} & 0 & 0 & \tfrac{1}{\rho}
 \end{pmatrix}
  \,\text{ and}
\]
\[
  \mathrm{D}G_1=
 \begin{pmatrix}
 0 & 1 & 0 & 0\\[1pt]
 -\tfrac{m_1^2}{\rho^2} & 2\tfrac{m_1}{\rho} & 0 & 0\\[4pt]
 -\tfrac{m_1m_2}{\rho^2} & \tfrac{m_2}{\rho} & \tfrac{m_1}{\rho} & 0\\[4pt]
 -\tfrac{m_1m_3}{\rho^2} & \tfrac{m_3}{\rho} & 0 & \tfrac{m_1}{\rho}
 \end{pmatrix}\,.
\]
Condition \eqref{eq:entr:integrcond} can then be verified by calculating the matrix products.
To calculate the entropy fluxes $Q_j$ we have to solve
\[
 \mathrm{D}Q_j=\mathrm{D}E\,\mathrm{D}G_j\,,
\]
and thus, for $j=1$, the system of equations
\begin{align*}
 \partial_\rho Q_1&=-\tfrac{m_1^2}{\rho^3}(m_1+m_2+m_3)\\
 \partial_{m_1} Q_1&=-\tfrac{1}{\rho^2}(3m_1^2+m_2^2+m_3^2)\\
 \partial_{m_2} Q_1&=-\tfrac{m_1m_2}{\rho^2}\\
 \partial_{m_3} Q_1&=\tfrac{m_1m_3}{\rho^2}\ .
\end{align*}
This system is, up to constants uniquely, solved by $Q_1=\tfrac{m_1}{2\rho^2}|\vec m|^2$.
Performing the same calculation for the other components of the flux, or using symmetry considerations, one can derive
\[
 Q_j=\tfrac{m_j}{2\rho^2}|\vec m|^2\,,
\]
and thus the entropy inequality
\begin{equation}\label{ineq:entr}
 \partial_t\left(\tfrac{|\vec m|^2}{2\rho}\right)+\sum_{j=1}^3\partial_{x_j}\left(\tfrac{m_j|\vec m|^2}{2\rho^2}\right)\leq 0\,.
\end{equation}
In the original variables this amounts to
\[
  \partial_t\left(\tfrac{1}{2}\rho|\vec u|^2\right)+\mbox{div}_{\vec x}\left(\vec u\,\tfrac{1}{2}\rho|\vec u|^2\right)\leq 0\,,
\]
and for radially symmetric solutions
\begin{equation}\label{ineq:entr:rad}
 \partial_t\left(\tfrac{1}{2}\rho u^2\right)+\partial_r\left(\tfrac{1}{2}\rho u^3\right)+\tfrac{n-1}{2r}\rho u^3\leq 0\,.
\end{equation}
\begin{definition}\label{def:diss}
Radially symmetric local shadow waves satisfying the entropy condition \eqref{ineq:entr:rad} in the weak limit, i.e.
\begin{equation*}
 \lim_{\eps\to 0}\Big(\partial_t\big(\tfrac{1}{2}\rho_{sw}^\eps (u_{sw}^\eps)^2\big)+\partial_r\big(\tfrac{1}{2}\rho_{sw}^\eps (u_{sw}^\eps)^3\big)+\tfrac{n-1}{2r}\rho_{sw}^\eps (u_{sw}^\eps)^3\Big)\leq 0
\end{equation*}
(when applied to nonnegative test functions) will be called \emph{dissipative shadow waves}.
\end{definition}

%\subsection{Shadow waves satisfying the entropy condition}\label{SWentrDisc}
\begin{proposition}\label{SWentrop}
A local shadow wave solution in the sense of Definition \ref{def:sw} to system \eqref{EulerNdRad} is dissipative if and only if
\begin{equation}\label{EntrCondkappa}
\kappa_1(u_0u_1-\dot c^2)\leq \kappa_2(u_0+u_1-2 \dot c)\,,
\end{equation}
or equivalently,
\begin{equation}\label{EntrCondGeneral}
-\dot c^3[\rho]+3\dot c^2[\rho u]-3\dot c[\rho u^2]+[\rho u^3]\leq0\,.
\end{equation}
\end{proposition}
\begin{proof}
First, algebraic transformations in system \eqref{EulerNdRad} immediately show that differentiable solutions satisfy the entropy condition \eqref{ineq:entr:rad} (actually with equality). Thus the regular parts $(\rho_0,u_0)$, $(\rho_1,u_1)$ of the shadow wave satisfy \eqref{ineq:entr:rad} in their respective domain. A calculation similar to the one leading to Proposition  \eqref{swprop} shows that
\begin{equation*}%\label{ineq:entr:sw}
-\dot c(t) [\rho u^2]+[\rho u^3]+ \partial_t\left(\sigma v^2\right)+\tfrac{n-1}{R + c(t)}\sigma v^3\leq 0\,;
\end{equation*}
here and in the sequel the jump terms are evaluated at $r = R + c(t)$.
For the first part a straightforward algebraic calculation yields
\begin{align*}
 -\dot c[\rho u^2]+[\rho u^3]=&-(u_0+u_1)\left(\dot c[\rho u]-[\rho u^2]\right)+u_0u_1\left(\dot c[\rho]-[\rho u]\right)\\
 =&-(u_0+u_1)\kappa_2+u_0u_1\kappa_1\,.
\end{align*}
Recalling that $v = \dot c$, we infer the entropy condition
\begin{equation*}
  \partial_t\left(\sigma \dot c(t)^2\right)+\tfrac{n-1}{R-c(t)}\sigma \dot c(t)^3\leq(u_0+u_1)\kappa_2-u_0u_1\kappa_1\,.
\end{equation*}
Using \eqref{equ:sw:def2} and applying the chain rule we can simplify this further to derive
\[
\dot c\kappa_2+\dot c\ddot c\sigma \leq(u_0+u_1)\kappa_2-u_0u_1\kappa_1\,.
\]
Invoking \eqref{equ:sw:def22} leads to the desired inequality \eqref{EntrCondkappa}. The second equivalent condition \eqref{EntrCondGeneral} follows by plugging in the definitions of $\kappa_1$ and $\kappa_2$.
\end{proof}

\subsection{The case of constant wave speed}
In this subsection we generally assume that either $\kappa_1$ or $\kappa_2$ is nonzero. (The case $\kappa_1 = \kappa_2 = 0$ was briefly discussed after Definition \ref{def:sw} and leads to a shock wave solution.) In particular, this also means that $\rho_0$ and $\rho_1$ cannot be zero simultaneously.
Applying \eqref{EntrCondkappa} to the special solutions of constant wave speed $\dot c = v_0$ and recalling that $\kappa_2=v_0\kappa_1$, we find the simpler form
\begin{equation}\label{EntrCondkappaCons}
 u_0u_1\kappa_1\leq (u_0+u_1-v_0)\kappa_2\,.
\end{equation}
\begin{proposition}\label{SWChoiceOfRoot}
The shadow wave solution of constant wave speed $v_0$ given in Proposition \ref{ConstWS:exist1} is physical and dissipative if and only if
\begin{equation}\label{Speed:entrSW}
 v_0=\,_1v_0=\frac{u_1\sqrt{\rho_1}+u_0\sqrt{\rho_0}}{\sqrt{\rho_1}+\sqrt {\rho_0}}\,.
\end{equation}
Further, if $\rho_0 > 0$ and $\rho_1 > 0$, then necessarily $u_0\geq u_1$ and the dissipativity condition \eqref{EntrCondkappaCons} is equivalent to the \emph{overcompressibility condition}
\begin{equation}\label{overcomp}
u_0\geq v_0 \geq u_1\,.
\end{equation}
\end{proposition}
\begin{proof}
\emph{The case $\kappa_2\neq0$:} In this case we have that $\kappa_1$ and $v_0$ are different from zero.
Thus $\kappa_1$ cannot change sign, and the special form of the solution \eqref{SwconstV} implies that $\kappa_1$ has to be positive
for the density $\rho$ to be nonnegative. (Note that depending on the sign of $v_0$, $R+v_0t$ is less or bigger than $R$.)
At any rate, $v_0$ and $\kappa_2$ have the same sign.
Using $\kappa_1=\kappa_2/v_0$ in the entropy condition we find
\[
 u_0u_1\tfrac{\kappa_2}{v_0}\leq (u_0+u_1-v_0)\kappa_2
\]
and upon multiplying by the positive quantity $v_0/\kappa_2$ and factorizing we conclude
\begin{equation}\label{prodcond}
 (u_0-v_0)(u_1-v_0)\leq0\,.
\end{equation}
Thus for a shadow wave to be physical (Definition~\ref{def:sw}) and dissipative (Definition~\ref{def:diss}) it is necessary that
\begin{equation}\label{entr:v:incl}
 \min\{u_0,u_1\}\leq v_0\leq \max\{u_0,u_1\}\,.
\end{equation}
Using the fact that $\kappa_1\geq 0$ equivalently yields
\begin{equation}\label{ruleoutineq}
  0\leq\kappa_1=v_0(\rho_1-\rho_0)-\rho_1u_1+\rho_0u_0=\rho_0(u_0-v_0) - \rho_1(u_1 - v_0)\,.
\end{equation}
Analyzing \eqref{prodcond} further, we see that the combination $u_0 - v_0 < 0$, $u_1 - v_0 > 0$ is ruled out by the nonnegativity of $\rho_0$ and $\rho_1$ and inequality \eqref{ruleoutineq}. The remaining possibilities are
\begin{itemize}
\item[(i)] $\rho_0 = 0$, $\rho_1 > 0$, $v_0 = u_1$\,;
\item[(ii)] $\rho_0 > 0$, $\rho_1 = 0$, $v_0 = u_0$\,;
\item[(iii)] $\rho_0 > 0$, $\rho_1 > 0$, $u_0-v_0\geq 0$ and $u_1 - v_0\leq 0$\,.
\end{itemize}
Case (iii) implies that $u_0\geq u_1$ and then the inequality $u_0\geq v_0 \geq u_1$ is equivalent with \eqref{prodcond} and in turn with \eqref{EntrCondkappaCons}.
As derived in Section \ref{sec:constspeed}, $v_0$ must be one of the roots \eqref{ConstWS:speeds}. The first root $v_0=\,_1v_0$ satisfies the inequality $u_0\geq v_0 \geq u_1$, as can be seen by direct calculation. The second root $_2v_0$ is ruled out by \eqref{entr:v:incl} and the lemma below (unless $u_0 = u_1$, in which case, however, $v_0 =\,_1v_0 =\,_2v_0$).
Thus the proposition is proven in case (iii).

In cases (i) and (ii), $v_0$ is also given by the formula \eqref{Speed:entrSW}.

\emph{The case $\kappa_2\equiv 0$:}
Here the entropy condition simplifies to $u_0u_1\kappa_1\leq0$. By our basic assumption in this section, $\kappa_1\neq 0$. Thus we must have that $v_0 = 0$, and it follows that $\kappa_1=-[\rho u]$ and
\begin{equation}\label{ineq:entr:case2}
u_0u_1(-[\rho u])\leq 0\,.
\end{equation}
Since the density in the shadow wave is given by $\sigma=-t[\rho u]$ we must assume that $[\rho u]\leq 0$ to obtain physical solutions. This together with \eqref{ineq:entr:case2}
yields that the signs of $u_0$ and $u_1$ have to be different. In addition, $-[\rho u]=-\rho_1u_1+\rho_0u_0\geq 0$. This rules out that $u_0 < 0$ and $u_1 > 0$ and we are again left with the three cases
\begin{itemize}
\item[(i)] $\rho_0 = 0$, $\rho_1 > 0$, $v_0 = 0 = u_1$;
\item[(ii)] $\rho_0 > 0$, $\rho_1 = 0$, $v_0 = 0 = u_0$;
\item[(iii)] $\rho_0 > 0$, $\rho_1 > 0$, $u_0\geq 0$ and $u_1\leq 0$.
\end{itemize}
In case (iii) we arrive at overcompressibility:
\begin{equation}\label{entr:Swrest}
u_0\geq v_0 = 0 \geq u_1\,.
\end{equation}
Finally, the equation $\kappa_2 = u_1^2\rho_1-u_0^2\rho_0=0$ together with \eqref{entr:Swrest} shows that in all cases $v_0 =\,_1v_0=0$. \end{proof}

\begin{lemma}
For all choices of real numbers $u_0 \neq u_1$ and $\rho_0 > 0$, $\rho_1 > 0$,  $\rho_0 \neq \rho_1$, we have
\begin{equation*}
 _2v_0=\frac{u_1\sqrt{\rho_1}-u_0\sqrt{\rho_0}}{\sqrt{\rho_1}-\sqrt {\rho_0}}\ \notin\ \big[\min\{u_0,u_1\}, \max\{u_0,u_1\}\big]\;.
\end{equation*}
\end{lemma}
\begin{proof}
To prove the assertion, we distinguish cases according to the signs of the denominator. In the first case,
$\sqrt{\rho_1}-\sqrt {\rho_0}<0$. Here we use
\begin{equation*}
 \frac{u_1\sqrt{\rho_1}-u_0\sqrt{\rho_0}}{\sqrt{\rho_1}-\sqrt {\rho_0}}=u_0+\frac{[u]\sqrt{\rho_1}}{\sqrt{\rho_1}-\sqrt {\rho_0}}\,.
\end{equation*}
Now if $u_0>u_1$ then $[u]<0$ and thus $_2v_0>u_0>u_1$. If on the other hand $u_0 < u_1$ then $[u] >0$ and thus $_2v_0< u_0< u_1$.

In the second case in which $\sqrt{\rho_1}-\sqrt {\rho_0}>0$, we rewrite
\begin{equation*}
 \frac{u_1\sqrt{\rho_1}-u_0\sqrt{\rho_0}}{\sqrt{\rho_1}-\sqrt {\rho_0}}=u_1+\frac{[u]\sqrt{\rho_0}}{\sqrt{\rho_1}-\sqrt {\rho_0}}\,.
\end{equation*}
For $u_0>u_1$ we again have $[u]<0$ and conclude that $_2v_0<u_1<u_0$, while for $u_0 < u_1$ and thus $[u] >0$ we infer that $_2v_0 > u_1 > u_0$.
\end{proof}

\section{The pseudo-Riemann problem}\label{SecRiem}

In the next step we aim at solving the Riemann problem for system \eqref{EulerNdRad}, together with the boundary conditions \eqref{boundcond}, by solutions that can contain radially symmetric shadow waves and point masses at the origin.

\subsection{Pseudo-Riemann data}

First we want to discuss which problem in spherical coordinates corresponds to the piecewise constant initial data of the Riemann problem in a single space dimension.
From equation \eqref{EulerNdRad} we calculate the solutions that are constant in time. Apart from the vacuum solutions $\rho=0$ for which $u$ can actually be an arbitrary function of $r$ and $t$ these are
\begin{equation}
\begin{split}
u=c_1,\ \rho=c_2\cdot r^{1-n}
\end{split}
\end{equation}
where $c_1, c_2$ are arbitrary constants.
This corresponds to solutions of equal mass on all spheres moving radially with velocity $u$ and will replace the constant solutions in our considerations of the Riemann problem.
Throughout this section we develop a complete set of shadow wave solutions for the following set of \emph{pseudo-Riemann} data
\begin{equation}\label{eq:pseudodada}
(\rho,u)|_{t=0}=\begin{cases} (\rho_{0},u_{0})=(\rho_{l}r^{1-n},u_{l}),\qquad r < R \\
(\rho_{1},u_{1})=(\rho_{r}r^{1-n},u_{r}),\qquad r > R\end{cases}
\end{equation}
where $R>0$ denotes the position of the jump in the initial data and $\rho_l, u_l$ as well as $\rho_r,u_r$ are constants satisfying the physical assumption $\rho_l,\rho_r\geq 0$.
Thus the subscript $l$ denotes the solution for $r$ smaller than the location of the jump---comparable to the left-hand side of the jump in the one-dimensional situation---and the
corresponding analogy holds for the subscript $r$. For notational simplicity in the diagrams we have set, in accordance with the previous sections,
$(\rho_0,u_0)=(\rho_l r^{1-n},u_l)$ and $(\rho_1,u_1)=(\rho_r r^{1-n},u_r)$.

\subsection{Uniqueness of entropy solutions for pseudo-Riemann data}
We consider pseudo-Riemann data of the form \eqref{eq:pseudodada}
and wish to show that any physical and dissipative shadow wave starts out propagating with constant velocity.
\begin{proposition}\label{ConstantSpeed}
Assume pseudo-Riemann initial data of the form \eqref{eq:pseudodada} where $\rho_{l},\rho_{r}> 0$ and $u_{l}>u_{r}$ are real constants. Then any physical and dissipative local shadow wave solution has a constant velocity $\dot{c}$. This holds true as long as no interaction with another wave takes place.
\end{proposition}

{\sc Remark.} Such an interaction can arise, for example, for $u_{l} > 0$, when a contact discontinuity separating the vacuum state emanating from the origin meets the shadow wave, cf. Figure~\ref{BildRiemannExtinct}.

\begin{proof}
Equality \eqref{equ:sw:def22} and the definition of $\kappa_1$, $\kappa_2$ yield
\begin{equation*}
\sigma \ddot{c}+\dot{c}\big(\dot{c}[\rho] - [\rho u]\big) = \dot{c}[\rho u]-[\rho u^{2}]\,.
\end{equation*}
By the same calculation as in \eqref{ConstWS:speeds} this can be written as
\begin{equation} \label{new2}
\sigma \ddot{c}+(\dot{c}-\,_1v_0)(\dot{c}-\,_2v_0)[\rho]=0
\end{equation}
with the velocities
\begin{equation*}
 _1v_0=\frac{u_1\sqrt{\rho_1}+u_0\sqrt{\rho_0}}{\sqrt{\rho_1}+\sqrt {\rho_0}}=\frac{u_r\sqrt{\rho_r}+u_l\sqrt{\rho_l}}{\sqrt{\rho_r}+\sqrt {\rho_l}}\quad \text{and}
 \quad _2v_0=\frac{u_r\sqrt{\rho_r}-u_l\sqrt{\rho_l}}{\sqrt{\rho_r}-\sqrt {\rho_l}}\,.
\end{equation*}
A priori, these velocities need not be constant. It is our aim to verify that in the physical, dissipative case these velocities are constant for pseudo-Riemann data as long as $\rho_r$ and $\rho_l$ at the position of the shadow wave do not change.
At time $t=0$ the shadow wave starts with density zero, so $\sigma(0)=0$. Equation \eqref{equ:sw:def22} leads to $\dot{c}(0)\kappa_1(0) = \kappa_2(0)$. Using the same calculations as in the proof of Proposition \ref{SWChoiceOfRoot}, dissipativity (or equivalently overcompressibility) holds initially if and only if
$$
\dot c(0)=\,_1v_0\,.
$$
Evaluated at $t=0$ the term $(\dot{c}-\,_2v_0)[\rho]$ can be calculated to be
\[
 (\dot{c}-\,_2v_0)[\rho]=(\,_1v_0-\,_2v_0)[\rho]=(u_l-u_r)\sqrt{\rho_l\rho_r}r^{1-n}>0\,.
\]
Further, $\sigma \geq 0$ because the shadow wave is assumed to be physical.

Suppose that $\dot{c}$ increases as $t>0$. Then $\sigma \ddot{c}\geq 0$
and $(\dot{c}-\,_1v_0)(\dot{c}-\,_2v_0)[\rho]>0$. That is, (\ref{new2})
cannot be satisfied. Similarly, if $\dot{c}$ decreases, both terms
are now negative and (\ref{new2}) cannot be satisfied either.
Therefore, $\dot{c}$ has to be constant, and the shadow wave is given by formula \eqref{SwconstV}.
%A simple calculation shows that $\kappa_1 = (u_l - u_r)\sqrt{\rho_l\rho_r}r^{1-n} > 0$. Hence the mass in the shadow wave is positive, and the solution is physical (and dissipative by Proposition \ref{SWChoiceOfRoot}).
\end{proof}
\begin{corollary}
Under the assumptions of Proposition \ref{ConstantSpeed}, the pseudo-Riemann problem for system \eqref{EulerNdRad} has at most one physical and dissipative local shadow wave solution, locally near $r = R$ and for sufficiently small time.
\end{corollary}
\begin{proof}
The assertion follows by combining Proposition \ref{ConstantSpeed} and Proposition \ref{SWChoiceOfRoot}.
\end{proof}
{\sc Remark.}
The condition $u_{l}>u_{r}$ corresponds to the case when a shadow wave solution of nonzero mass will develop. Thus we have derived that physical shadow wave solutions of the Riemann problem
with nonincreasing kinetic energy have a unique wave speed.
This is in general not the case for nonregular solutions of pressureless gas dynamics (cf.~\cite{Bo94}); not even using all smooth convex entropy functionals can remedy this nonuniqueness.
Our result is however in line with these considerations because we only admit a single shadow wave. If two shadow waves starting from the same point are allowed, it is actually very simple to construct examples for nonuniqueness. However, multiple shadow waves can be ruled out either by \emph{requiring} overcompressibility or by resorting to the underlying sticky particle system (cf. \cite{BrGr}).
%In our case this amounts to assuming the minimal number (namely zero or one) of shadow waves in the solution.

\subsection{Existence of global solutions to the pseudo-Riemann problem}

In this section, we will compute radially symmetric, physical and dissipative solutions to the pseudo-Riemann problem for system \eqref{EulerNdRad} explicitly. The solutions will exist globally in time (i.e., for $t\geq 0$) and will be composed of a minimal number of shadow waves, contact discontinuities, vacuum states bounded by shock spheres, and possibly a delta function singularity accumulating at the origin. The conditions imposed at $r=0$ are the ones spelled out in Definition \ref{def:radialSW}.

Recall from the remark after Definition \ref{def:sw} that the Rankine-Hugoniot conditions are given by $\kappa_1 = \kappa_2 = 0$. As noted there, these conditions can only hold if $\rho_0 =0$ or $\rho_1 = 0$ or $u_0 = u_1$. The first two cases lead to a vacuum state bounded by a shock curve with speed $\dot{c} = u_1$ or $\dot{c} = u_0$, respectively. The third case amounts to a contact discontinuity. All three waves arise as special cases of shadow waves for which $\sigma(t) = 0$.

\begin{proposition}
Let $\rho_0, \rho_1 \geq 0$ and let $u_l, u_r$ be arbitrary real numbers. Then the pseudo-Riemann problem \eqref{eq:pseudodada} for system \eqref{EulerNdRad} has a radially symmetric, global, physical and dissipative solution. At each point of time, the solution consists of at most one shadow wave, one contact discontinuity, three shocks following vacuum states and a delta function at the origin.
\end{proposition}

The remainder of this subsection will be devoted to proving the proposition.
We will distinguish cases depending on the signs of $u_l$ and the relation between $u_l$ and $u_r$.

\subsubsection{Pseudo-Riemann data, first case: $u_l\leq 0$}

\paragraph{The case $u_r> u_l$: vacuum state and shocks}
In this case the parts of the solution separate at the jump and the density is zero in-between. The velocity in that region can be chosen arbitrarily since in the equations $\rho$ always multiplies the velocity and all jump conditions are trivially fulfilled. The following graphs illustrate the situation in case $u_r$ is larger or smaller than zero.
We start with the case $u_r>0$, depicted in Figure~1.
\begin{figure}[htb]
\centering
\begin{tikzpicture}[scale=0.9]
% horizontal axis
\draw[->] (0,0) -- (6,0) node[anchor=north] {$r$};
% labels
\draw	(1,0) node[anchor=north] {$\rho_0$, $u_0$}
		(3,0) node[anchor=north] {}
		(5,0) node[anchor=north] {$\rho_1$, $u_1$};
\draw[->, dashed] (4,0)--(5,2);
\draw[->, dashed] (2,0)--(1.5,2);
% ranges
%\draw	(1,3.5) node{{\scriptsize Constant flux}}
%		(4,3.5) node{{\scriptsize Field weakening}};

% vertical axis
\draw[->] (0,0) -- (0,4) node[anchor=east] {$t$};
% nominal speed
\draw[dotted] (3,0) -- (5,4);
\draw[dotted] (3,0) -- (2,4);
\draw (5.5,3) node[anchor=north] {$\rho_1$,$u_1$}
      (1,3) node[anchor=north] {$\rho_0$,$u_0$}
      (3.2,3) node[anchor=north] {$\rho=0$};
\end{tikzpicture}
\caption{Vacuum state, $u_l<0\leq u_r$}
\end{figure}
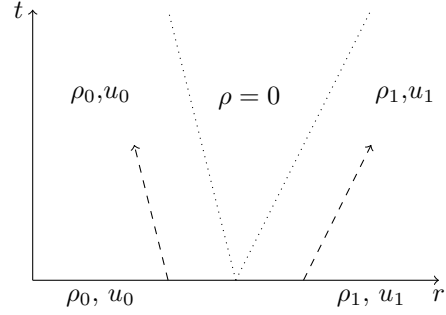

For $u_l<0$, the complete solution is given by
\begin{equation*}
\begin{split}
m_0(t)&=\left\{
 \begin{aligned}
 &-t\rho_l u_l|\mathbb{S}^{n-1}| \quad &&\text{for }0< t\leq -\tfrac{R}{u_l}\\
 &R\rho_l|\mathbb{S}^{n-1}| \quad && -\tfrac{R}{u_l}\leq t\,,
 \end{aligned}
 \right.\\[4pt]
(\rho,u)&=\left\{
 \begin{aligned}
 &(\rho_l r^{1-n},u_l)\quad &&\text{for } 0< r\leq R+u_l t\\
 &(0,\text{arb.}) \quad&& R+u_l t\leq r\leq R+u_r t\\
 &(\rho_r r^{1-n},u_r) \quad && R+u_rt\leq r\,.
 \end{aligned}
 \right.
\end{split}
 \end{equation*}
For $u_l=0$ the same holds true but $m_0(t)=0$. In case $u_r< 0$ we have the situation of Figure~2.\\
\begin{figure}[htb]
\centering
\begin{tikzpicture}[scale=0.9]
% horizontal axis
\draw[->] (0,0) -- (6,0) node[anchor=north] {$r$};
% labels
\draw	(1,0) node[anchor=north] {$\rho_0$, $u_0$}
		(3,0) node[anchor=north] {}
		(5,0) node[anchor=north] {$\rho_1$, $u_1$};
% ranges
%\draw	(1,3.5) node{{\scriptsize Constant flux}}
%		(4,3.5) node{{\scriptsize Field weakening}};
% vertical axis
\draw[->] (0,0) -- (0,4) node[anchor=east] {$t$};
% nominal speed
\draw[->, dashed] (1.2,0)--(0.2,0.667);
\draw[->, dashed] (5,0)--(4.43,0.667);
\draw[dotted] (3,0) -- (0,2);
\draw[dotted] (3,0) -- (0,3.5);
\draw (3,3) node[anchor=north] {$\rho_1$,$u_1$}
      (1,1) node[anchor=north] {$\rho_0$,$u_0$}
      (0.6,2.5) node[anchor=north] {$\rho=0$};
\end{tikzpicture}
 \caption{Vacuum state, $u_l<u_r<0$}
\end{figure}
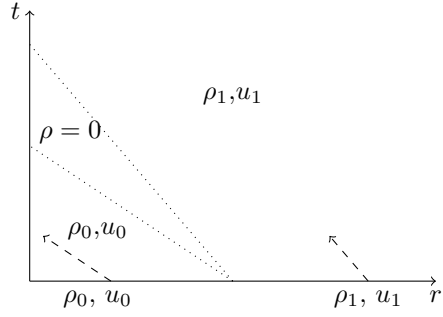

The complete solution is a bit more complicated at $0$,
\begin{equation*}
\begin{split}
m_0(t)&=\left\{
 \begin{aligned}
 &-t\rho_l u_l|\mathbb{S}^{n-1}|\quad &&\text{for } 0< t\leq -\tfrac{R}{u_l}\\
 &R\rho_l|\mathbb{S}^{n-1}|\quad && -\tfrac{R}{u_l}\leq t\leq -\tfrac{R}{u_r}\\
 &(R\rho_l-(t+\tfrac{R}{u_r})\rho_ru_r)|\mathbb{S}^{n-1}|\quad && -\tfrac{R}{u_r}\leq t\,,
 \end{aligned}
 \right.\\[4pt]
(\rho,u)&=\left\{
 \begin{aligned}
 &(\rho_l r^{1-n},u_l) \quad &&\text{for }0< r\leq R+u_l t\\
 &(0,\text{arb.})\quad && R+u_l t\leq r\leq R+u_r t\\
 &(\rho_r r^{1-n},u_r)\quad && R+u_rt\leq r\,.
 \end{aligned}
 \right.
\end{split}
 \end{equation*}

\paragraph{The case $u_r=u_l$: contact discontinuity}
The contact discontinuity solution can be viewed as a special case of the solutions above, with width of the lacuna equal to zero. Thus we will not write down the solution again.

\paragraph{The case $u_r<u_l$: shadow wave}
We have learned that a physical and dissipative shadow wave only exists in this case. As determined by \eqref{Speed:entrSW} it travels with velocity
\begin{equation*}
v_0= \frac{u_1\sqrt{\rho_1}+u_0\sqrt{\rho_0}}{\sqrt{\rho_1}+\sqrt {\rho_0}}=\frac{u_r\sqrt{\rho_r}+u_l\sqrt{\rho_l}}{\sqrt{\rho_r}+\sqrt {\rho_l}}
\end{equation*}
which satisfies the overcompressibility condition $u_r<v_0<u_l$, cf. Figure~3.

\begin{figure}[htb]
\centering
 \begin{tikzpicture}[scale=0.9]
% horizontal axis
\draw[->] (0,0) -- (6,0) node[anchor=north] {$r$};
% labels
\draw	(1,0) node[anchor=north] {$\rho_0$, $u_0$}
		(3,0) node[anchor=north] {}
		(5,0) node[anchor=north] {$\rho_1$, $u_1$};
% ranges
%\draw	(1,3.5) node{{\scriptsize Constant flux}}
%		(4,3.5) node{{\scriptsize Field weakening}};

% vertical axis
\draw[->] (0,0) -- (0,4) node[anchor=east] {$t$};
% nominal speed
\draw[dotted] (3,0) -- (0,3.5);
%\draw[dotted] (3,0) -- (2,4);
\draw (4,3) node[anchor=north] {$\rho_1$,$u_1$}
      (1,1) node[anchor=north] {$\rho_0$,$u_0$}
      (1.1,2.5) node[anchor=north] {SW};
\draw[->, dashed] (2,0) -- (1.7,1.2);
\draw[->, dashed] (4,0) -- (2.3,1.2);
\end{tikzpicture}
\caption{Shadow wave, $u_r<u_l\leq 0$}
\end{figure}
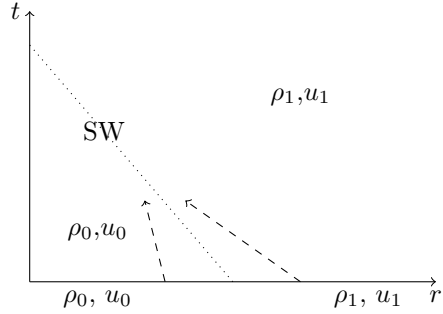
In order to calculate the mass in the shadow wave we need to determine the influx, given by $\kappa_1$.
A straightforward calculation yields
\begin{equation*}
 \kappa_1=\sqrt{\rho_0}\sqrt{\rho_1}(u_0-u_1)=r^{1-n}\sqrt{\rho_l} \sqrt{\rho_r}(u_l-u_r)\,.
\end{equation*}
In the case $\kappa_2\neq 0$, we can derive the mass $\sigma$ in the shadow wave from equation \eqref{SWMass} and it is given by
\begin{equation}\label{sigt}
 \sigma(t)=t\sqrt{\rho_l}\sqrt{\rho_r}(u_l-u_r)(R+v_0t)^{1-n}\,.
\end{equation}
With this result we can again give a complete representation of the shadow wave solution,
\begin{equation*}
\begin{split}
m_0(t)&=-|\mathbb{S}^{n-1}|\rho_lu_l t+H(t-t_0)\big(\sqrt{\rho_l} \sqrt{\rho_r}[u]t_0 + |\mathbb{S}^{n-1}|[\rho u]t_0
- |\mathbb{S}^{n-1}| \rho_ru_r t\big)\,,\\[4pt]
(\rho,u)&=\left\{
 \begin{aligned}
 &(\rho_l r^{1-n},u_l) &&0< r\leq R+v_0 t-\tfrac{\eps}{2}\\
 &\big(t\sqrt{\rho_l}\sqrt{\rho_r}(u_l-u_r)\tfrac{1}{\eps},v_0\big) && R+v_0 t-\tfrac{\eps}{2}\leq r\leq R+v_0 t+\tfrac{\eps}{2}\\
 &(\rho_r r^{1-n},u_r) && R+v_0t+\tfrac{\eps}{2}\leq r
 \end{aligned}
 \right.
\end{split}
 \end{equation*}
with $t_0 = - R/v_0$, where $H$ denotes the Heaviside function.

When $\kappa_2 = 0$, either $\kappa_1 = 0$ as well, or $v_0 = 0$. Either possibility can only arise when one of the initial states is a vacuum state. Thus this case leads to a vacuum state followed by a shock (or vice versa).

\subsubsection{Pseudo-Riemann data, second case: $u_l>0$}

In case $u_l>0$ initially we always have a vacuum state followed by a shock emanating from $\vec x=0$. The situation at $|\vec x|=R$ depends on the relation of $u_l$ and $u_r$.

\paragraph{The case $u_r\geq u_l$: vacuum states or contact discontinuity}
In this situation we either have another vacuum state (for $u_r>u_l$) or a contact discontinuity (for $u_r=u_l$) evolving from $r=R$ at $t=0$. The vacuum case is depicted in Figure 4.

\begin{figure}[htb]
\centering
\begin{tikzpicture}[scale=0.9]
% horizontal axis
\draw[->] (0,0) -- (6,0) node[anchor=north] {$r$};
% labels
\draw	(1,0) node[anchor=north] {$\rho_0$, $u_0$}
		(3,0) node[anchor=north] {}
		(5,0) node[anchor=north] {$\rho_1$, $u_1$};
\draw[->, dashed] (4,0)--(5.5,2);
\draw[->, dashed] (1.5,0)--(2,2);
\draw[->] (0,0) -- (0,4) node[anchor=east] {$t$};
% nominal speed
\draw[dotted] (0,0) -- (1,4);
\draw[dotted] (3,0) -- (6,4);
\draw[dotted] (3,0) -- (4,4);
\draw (5.8,3) node[anchor=north] {$\rho_1$,$u_1$}
      (2.2,3) node[anchor=north] {$\rho_0$,$u_0$}
      (4.3,3) node[anchor=north] {$\rho=0$}
      (0.2,3) node[anchor=north] {$\rho=0$};
\end{tikzpicture}
\caption{Vacuum states, $0<u_l<u_r$}
\end{figure}
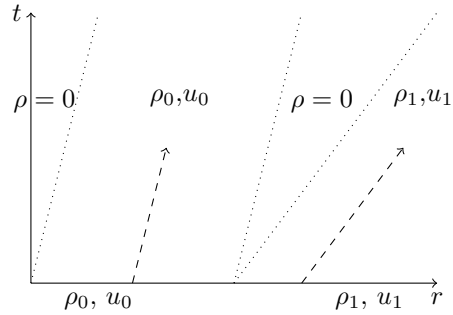
The explicit representation of the solution takes the simple form
\begin{equation*}
\begin{split}
m_0(t)&=0\,,\\
(\rho,u)&=\left\{
 \begin{aligned}
 &(0,\text{arb}) \quad &&\text{for } 0<r< u_l t\\
  &(\rho_l r^{1-n},u_l)\quad && u_l t< r\leq R+u_l t\\
 &(0,\text{arb.}) \quad&& R+u_l t\leq r\leq R+u_r t\\
 &(\rho_r r^{1-n},u_r) \quad && R+u_lt\leq r\,.
 \end{aligned}
 \right.
\end{split}
 \end{equation*}

\paragraph{The case $u_r<u_l$: shadow wave} We begin with the generic case $\rho_l > 0$, $\rho_r > 0$.
In this situation the overcompressibility condition $u_l\geq v_0\geq u_r$ must be satisfied.
This is the most delicate case since at some point the shadow wave will have absorbed all the mass on its left-hand side
(more precisely, the mass in the interior of the sphere of radius $R+v_0t$)
as illustrated in Figure~5.

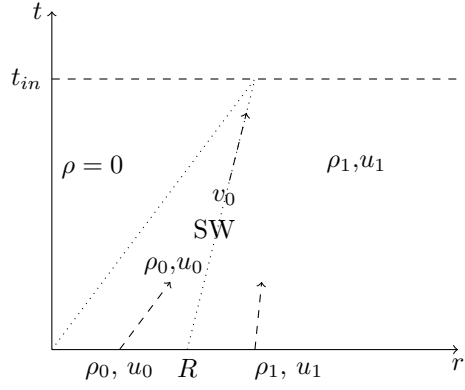
\begin{figure}[htb]
\centering
\begin{tikzpicture}[scale=0.9]
% axis
\draw[->] (0,0) -- (6,0) node[anchor=north] {$r$};
\draw[->] (0,0) -- (0,5) node[anchor=east] {$t$};
% labels
\draw	(1,-0.05) node[anchor=north] {$\rho_0$, $u_0$}
		(2,0) node[anchor=north] {}
		(3.5,-0.05) node[anchor=north] {$\rho_1$, $u_1$};
\draw[->, dashed] (3,0)--(3.1,1);
\draw[->, dashed] (1,0)--(1.75,1);
\draw[dotted] (0,0) -- (3,4);
%\draw[dotted] (3,0) -- (6,4);
\draw[dotted] (2,0) -- (3,4);
\draw[dashed] (0,4) -- (6,4);
\draw[->, dashed] (2.625,2.5)--(2.875,3.5);
\draw (4.5,3) node[anchor=north] {$\rho_1$,$u_1$}
      (1.8,1.5) node[anchor=north] {$\rho_0$,$u_0$}
      (0.6,3) node[anchor=north] {$\rho=0$}
       (2.55,2.5) node[anchor=north] {$v_0$}
       (2.4,2.05) node[anchor=north] {\text{SW}}
       (2,0) node[anchor=north] {$R$}
      (0,4) node[anchor=east]{$t_{in}$};
\end{tikzpicture}
\caption{Shadow wave with pseudo-Riemann data and extinct left-hand side solution}\label{BildRiemannExtinct}
\end{figure}

For $t<t_{in}$ the solution is the same as in the previous shadow wave case when $u_r<u_l$.
We have to calculate the solution after the interior part of the initial data has been absorbed completely, which will happen at time
\begin{equation*}
t_{in}=\frac{R}{u_l-v_0}=R\frac{\sqrt{\rho_r}+\sqrt{\rho_l}}{\sqrt{\rho_r}(u_l-u_r)}\,.
\end{equation*}
In this situation we are left with a shadow wave with nonzero initial mass and vacuum in its interior.
The corresponding initial time is $t_{in}$ and initial mass and velocity are given by the values in the shadow wave at $t_{in}$.
Equation \eqref{sigt} evaluated at $t_{in}$ yields
\begin{equation*}
\sigma(t_{in})=t_{in}\sqrt{\rho_l}\sqrt{\rho_r}(u_l-u_r)\,r(t_{in})^{1-n}
    = R\sqrt{\rho_l}(\sqrt{\rho_l}+\sqrt{\rho_r})\,r(t_{in})^{1-n}
\end{equation*}
with $r(t_{in}) = R+v_0t_{in}$. Further, $v_0$ is again given by
\begin{equation}\label{swveloctotin}
 v_0= \frac{u_r\sqrt{\rho_r}+u_l\sqrt{\rho_l}}{\sqrt{\rho_r}+\sqrt {\rho_l}}\,,
\end{equation}
and the position at that time is
\[
r(t_{in})=R\left(1+\frac{u_r\sqrt{\rho_r}+u_l\sqrt{\rho_l}}{\sqrt{\rho_l}(u_l-u_r)}\right)
   =R\frac{u_l(\sqrt{\rho_r}+\sqrt{\rho_l})}{\sqrt{\rho_r}(u_l-u_r)}\,.
\]
For times larger than $t_{in}$ we have to calculate the evolution of a shadow wave with initial nonzero mass and momentum, but for $\rho_l=0$.
To do so observe that for $\rho_l=\rho_0=0$ we are in the simple case $\kappa_1=\rho_1(\dot c-u_1)=r^{1-n}\rho_r(\dot c-u_r)$ and $\kappa_2=u_r\kappa_1$.
Since $\sigma(t_{in})\neq 0$ we expect the system of differential equations \eqref{equ:sw:def1}, \eqref{equ:sw:def2}
to have a (locally) unique solution $\sigma(t), \dot c(t)$ provided we prescribe the initial mass and velocity of the delta function part.
Using the special form of $\rho_1$ in the system of equations and performing the substitution $\sigma r^{n-1}=p$ (with $r = R + c(t)$) yields
\begin{align*}
 \dot p&=\rho_r(\dot c-u_r)\\
 \ddot c p+\dot c\dot p&=u_r\rho_r(\dot c-u_r)\,.
\end{align*}
Plugging the first equation into the second one and denoting the relative velocity by $w=\dot c-u_r$ results in the equation
\[
p=-\rho_r\frac{w^2}{\dot w}\,.
\]
Inserting this in the first equation we derive
\[
 \frac{\ddot w}{\dot w}=3\frac{\dot w}{w}\,.
\]
This equation can be integrated once readily to yield
\begin{equation}\label{omegadiff}
 \frac{\dot w}{w^3}=-\tfrac{1}{2}C\,,
\end{equation}
for a constant $C$. (We choose the minus sign expecting $C$ to be nonnegative because $w$ is expected to be positive and $\dot w$ negative.) Another integration results in
\[
 w=\left(Ct+D\right)^{-1/2}\,,
\]
where $D$ is an arbitrary constant. Using equation~\eqref{omegadiff} and the result above in the equation for $p$ we arrive at
\[
 p=-\rho_r\frac{w^2}{\dot w}=\rho_r\frac{2}{Cw}=\frac{2\rho_r}{C}\left(Ct+D\right)^{1/2}\,.
\]
In this way the unique solution can be calculated for $t>t_{in}$. In the original variables it reads
\begin{align}
 \dot c(t)&=u_r+(Ct+D)^{-1/2} \notag\\
 \sigma(t)&=\frac{2\rho_r}{C}\left(Ct+D\right)^{1/2}\big(R+c(t)\big)^{1-n}\,.\label{sigmaaftertin}
\end{align}
We integrate the first equation in $t$ again:
\begin{equation}\label{caftertin}
 c(t)=u_rt+E+\frac2{C}(Ct+D)^{1/2}\,,
\end{equation}
and can determine the constants $C,\,D,\,E$ from the conditions at $t=t_{in}$, namely
\begin{align*}
&\dot c(t_{in})=u_r+(Ct_{in}+D)^{-1/2}=\frac{u_r\sqrt{\rho_r}+u_l\sqrt{\rho_l}}{\sqrt{\rho_r}+\sqrt {\rho_l}}\\
&\sigma(t_{in})=\frac{2\rho_r}{C}\left(Ct_{in}+D\right)^{1/2}\,r(t_{in})^{1-n}=R\sqrt{\rho_l}(\sqrt{\rho_r}+\sqrt{\rho_l})\,r(t_{in})^{1-n}\\
&c(t_{in})=u_rt_{in}+E+\frac2{C}(Ct_{in}+D)^{1/2}=R\frac{u_l(\sqrt{\rho_r}+\sqrt{\rho_l})}{\sqrt{\rho_r}(u_l-u_r)}\,.
\end{align*}
This system of equations can be solved for $C,\,D$ and $E$ and we arrive at
\begin{align}
 C=&\frac{2\rho_r}{R\rho_l(u_l-u_r)}\label{constantsC}\\
 D=&\frac{\rho_l-\rho_r}{\rho_l(u_l-u_r)^2}\label{constantsD}\\
 E=&\frac{R}{\rho_r}(\rho_r - \rho_l)\label{constantsE}
\end{align}
The velocity of the shadow wave $\dot{c}(t)$ will approach $u_r$ for $t\rightarrow\infty$. If $u_r\geq 0$ the shadow wave will never reach zero and we are in the situation of Figure~6.
\begin{figure}[htb]
\centering
\begin{tikzpicture}[scale=0.9]
% axis
\draw[->] (0,0) -- (6,0) node[anchor=north] {$r$};
\draw[->] (0,0) -- (0,6) node[anchor=east] {$t$};
% labels
\draw	(1,0) node[anchor=north] {$\rho_l$, $u_l$}
		(2,0) node[anchor=north] {}
		(5,0) node[anchor=north] {$\rho_r$, $u_r$};
\draw[->, dashed] (3,0)--(3.1,1);
\draw[->, dashed] (1,0)--(1.75,1);
\draw[dotted] (0,0) -- (3,4);
%\draw[dotted] (3,0) -- (6,4);
\draw[dotted] (2,0) -- (3,4);
\draw[dotted] (3,4) .. controls (3.125,4.5) and (3.2,5) .. (3.3,6);
\draw (4.5,3) node[anchor=north] {$\rho_r$,$u_r$}
      (1.8,1.5) node[anchor=north] {$\rho_l$,$u_l$}
      (0.6,3) node[anchor=north] {$\rho=0$}
      (2.5,2) node[anchor=north] {\text{SW}}
      (3.2,5) node[anchor=north] {\text{SW}};
\end{tikzpicture}
\caption{Shadow wave, $0<u_r<u_l$}
\end{figure}
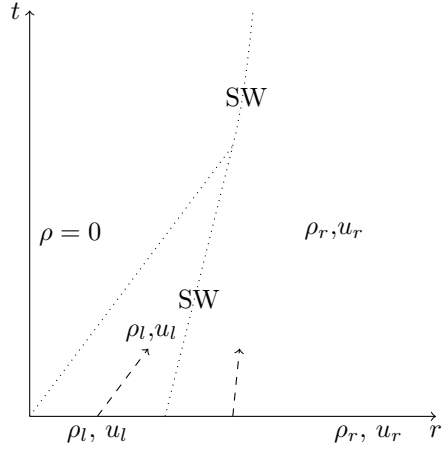

We do not write the complete solution but note that before $t_{in}$ the shadow wave position is determined by $R+v_0t$ with velocity from equation \eqref{swveloctotin}. The mass is given in \eqref{sigt}.
After $t_{in}$ the position of the shadow wave can be calculated from \eqref{caftertin} and the mass from \eqref{sigmaaftertin} by plugging in the constants given in \eqref{constantsC}--\eqref{constantsE}.
The mass at zero is identically zero, $m_0(t)\equiv 0$.

\begin{figure}[htb]
\centering
\begin{tikzpicture}[scale=0.9]
% axis
\draw[->] (0,0) -- (6,0) node[anchor=north] {$r$};
\draw[->] (0,0) -- (0,6) node[anchor=east] {$t$};
% labels
\draw	(1,0) node[anchor=north] {$\rho_l$, $u_l$}
		(2,0) node[anchor=north] {}
		(5,0) node[anchor=north] {$\rho_r$, $u_r$};
\draw   (0,5.5) node[anchor=east] {$t_{sw0}$};
\draw[->, dashed] (5,0)--(3,1);
\draw[->, dashed] (1,0)--(1.715,1);
\draw[dotted] (0,0) -- (2.5,3.5);
%\draw[dotted] (3,0) -- (6,4);
\draw[dotted] (2,0) -- (2.5,3.5);
\draw[dotted] (2.5,3.5) .. controls (2.8,4) and (1.5,5) .. (0,5.5);
\draw (4.5,3) node[anchor=north] {$\rho_r$,$u_r$}
      (1.6,1.5) node[anchor=north] {$\rho_l$,$u_l$}
      (0.6,3) node[anchor=north] {$\rho=0$}
      (2.3,2) node[anchor=north] {\text{SW}}
      (1.7,5) node[anchor=north] {\text{SW}};
\end{tikzpicture}
\caption{Shadow wave, $u_r<0<u_l$}
\end{figure}
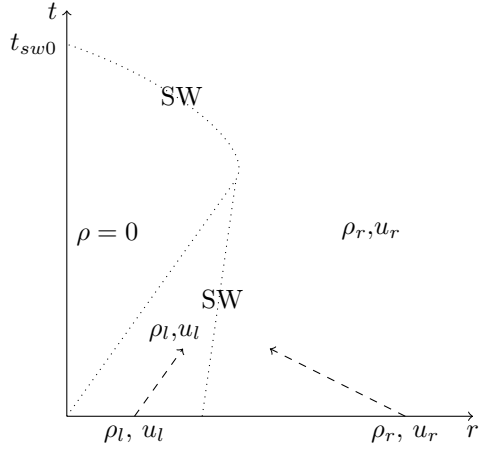
If $u_r<0$ then the shadow wave will eventually reach $r=0$, no matter whether $v_0$ is positive or negative. Figure~7 shows the case when $v_0>0$.
The shadow wave will reach $r=0$ at time $t_{sw0}$. Until this time the position of the shadow wave and the mass content can be calculated in exactly the same way as in the above case.
Now $c(t)$ from \eqref{caftertin} will have a zero and this determines $t_{sw0}$. The explicit formula for $t_{sw0}$ is rather complicated and thus we omit it. The mass at zero will initially be zero,
then at $t_{sw0}$ it will increase instantaneously by the amount of mass in the shadow wave $\sigma(t_{sw0})$.
For $t>t_{sw0}$ the mass at zero will grow linearly with the continuing influx $(\rho_r,u_r)$.

Finally, the remaining cases concern initial pseudo-Riemann data such that $\rho_l > 0$, $\rho_r = 0$, $u_r<u_l$ or $\rho_l = 0$, $\rho_r > 0$, $u_r<u_l$. The first case leads to vacuum states emanating from $\vec{x} = 0$ and $|\vec{x}| = R$ with velocity $v_0 = u_l$ (see the arguments in the proof of Proposition \ref{SWChoiceOfRoot}). The second case leads to a vacuum state bounded by a shock with velocity $v_0 = u_r$. If $u_r < 0$, mass will start accumulating at $\vec{x} = 0$ after time $t = -R/u_r$. The amount of mass can be computed from formula \eqref{boundcond}.

\subsection {Non-entropic solutions for pseudo-Riemann data}
For completeness we give an explicit example of a shadow wave solution for pseudo-Riemann initial data with nonconstant speed
$\dot c(t)$ in dimension two.
Such a solution cannot satisfy the entropy condition and be physical at the same time, as we showed earlier.
While construction of non-physical solutions is simple, the construction of a global, physical non-entropy solution is somewhat more difficult.

Consider Equation \eqref{EulerNdRad} in two dimensions and with initial data
\[
\rho(r,0)=\left\{\begin{aligned}
&0 & &r\leq1 \\
&1/r&  &r>1
\end{aligned}\right. \quad\text{and}\quad
u(r,0)=\left\{\begin{aligned}
&-1 & &r\leq1\\
&0&  &r>1\,.
\end{aligned}\right.
\]
We fix the position $r(t) = R + c(t)$ of the shadow wave by setting
\[ R=1 \text{ and }c(t)=\frac{t}{\sqrt{t+1}}\text{ , hence }r(t)=1+\frac{t}{\sqrt{t+1}}\]
and also the mass in the shadow wave as
\[
\sigma(t)=\frac{c(t)}{2r(t)}=\frac{t}{2(t+\sqrt{t+1})}\ ,
\]
which is $0$ for $t=0$ but becomes positive during evolution.
Since $r(t)$ is monotonically increasing and $u_r=0$ we know the right-hand side data at the position of the shadow wave. It is just $\rho_r=1/r(t)$ and $u_r=0$.
From the equations \eqref{equ:sw:def1}, \eqref{equ:sw:def2} for $n=2$ and $\ddot c\neq0$ we have
\begin{align*}
&\dot \sigma+\tfrac{1}{R+c(t)}\dot c \sigma =\kappa_1=\dot c[\rho]-[\rho u]\,,\\
&\sigma\ddot c+\dot c\kappa_1=\kappa_2=\dot c[\rho u]-[\rho u^2]\,.
\end{align*}
This system can be solved for $\rho_l$ and $u_l$, resulting in
\[
\rho_l(t)=\frac{(t+2)^2\left(t\sqrt{t+1}-t-1\right)}{2(t^2-t-1)(t^2+4t+8)}\quad\text{and}\quad u_l(t)=-\frac{2}{(t+1)^{3/2}(t+2)}\,.
\]
Note that, for positive $t$, $\rho_l$ is nonnegative and regular ($t=(1+\sqrt{5})/2$ is a common root of numerator and denominator), and $u_l$ is negative and increasing.
Thus the solution to the left of the location of the shadow wave can be constructed, e.g., by the method of characteristics. The complete solution is depicted in Figure~8.
\begin{figure}[htb]
\centering
\begin{tikzpicture}[scale=0.9]
% axis
\draw[->] (0,0) -- (6,0) node[anchor=north] {$r$};
\draw[->] (0,0) -- (0,6) node[anchor=east] {$t$};
\draw (1,1pt)--(1,-3pt) node[anchor=north] {$1$};
% labels
\draw	(0.2,0) node[anchor=north] {$\rho_0$, $u_0$}
		(2,0) node[anchor=north] {}
		(4,0) node[anchor=north] {$\rho_1$, $u_1$};
\draw[->, dashed] (3,0)--(3,1);
\draw[->, dashed] (4,0)--(4,1);
\draw[->, dashed] (5,0)--(5,1);
\draw[->, dashed] (0.5,0)--(0,0.5);
 \draw[scale=1,domain=0:6,dotted,variable=\x] plot ({\x/(\x+1)^(1/2)+1},{\x});
\draw[->, dashed] (1.40825,0.5)--(0.97278,1.5);
\draw[->, dashed] (1.707,1)--(1.471,2);
\draw[->, dashed] (2.1547,2)--(2.0585,3);
\draw[->, dashed] (2.5,3)--(2.45,4);
\draw[->, dashed] (2.7889,4)--(2.7591,5);
\draw[dotted] (1,0) -- (0,1);
%\draw[dotted] (3,0) -- (6,4);
%\draw[dotted] (2,0) -- (3,4);
%\draw[dotted] (3,4) .. controls (3.125,4.5) and (3.2,5) .. (3.3,6);
\draw %(4.5,3) node[anchor=north] {$\rho_1$,$u_1$}
      %(1.8,1.5) node[anchor=north] {$\rho_0$,$u_0$}
      %(0.6,3) node[anchor=north] {$\rho=0$}
      (3.2,5.5) node[anchor=north] {\text{SW}};
\end{tikzpicture}
\caption{Physical shadow wave solution that does not satisfy the entropy condition}
\end{figure}
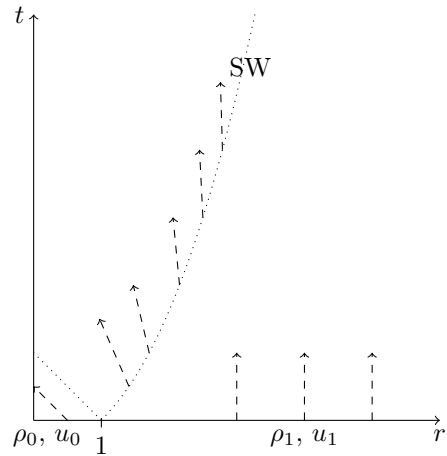

{\sc Remark.}
(a) Integrating the mass arriving at zero over time gives a complete solution; we omit the tedious calculations as the example is intended as an illustration.

(b) The shadow wave solution will not change its shape if $\rho(r,0)$ and $u(r,0)$ are chosen differently for $r<1$ as long as $u(r,0)$ is monotonically increasing in the interval $r\in]0,1[$ and $u(1,0)\leq -1$.

\end{document}